\documentclass[a4paper,11pt,reqno]{amsart}
\usepackage{amscd,amssymb,graphicx,color,epigraph}
\usepackage[a4paper, left=1in, right=1in, top=1in, bottom=1in]{geometry}
\usepackage[shortalphabetic,initials]{amsrefs}

\usepackage{enumitem}
\setlist[itemize]{leftmargin=18pt}
\setlist[enumerate]{leftmargin=18pt}

\usepackage[normalem]{ulem}

\usepackage{amssymb}
\usepackage{amsthm}
\usepackage{bm}
\usepackage{latexsym}
\usepackage{xcolor}
\usepackage{hyperref}
\usepackage{eufrak}
\usepackage{gensymb}
\usepackage{mathrsfs}
\usepackage{amscd}
\usepackage[all,cmtip]{xy}
\usepackage{color}
\usepackage{amscd}
\usepackage{listings}
\usepackage[center]{caption}
\usepackage{tikz-cd}
\usepackage{tikz-network}
\usepackage{float}
\theoremstyle{plain}
\usepackage{array, makecell} %
\usepackage{amsmath}
\usepackage{mathtools}

\numberwithin{equation}{section}
\hypersetup{
  colorlinks,
  citecolor=red,
  linkcolor=blue,
  urlcolor=magenta}
\newtheorem{theorem}{Theorem}[section]
\newtheorem{proposition}[theorem]{Proposition}

\newtheorem{lemma}[theorem]{Lemma}

\newtheorem{corollary}[theorem]{Corollary}

\newtheorem{remark}{Remark}

\theoremstyle{definition}

\newcommand{\appsection}[1]{\let\oldthesection\thesection
\renewcommand{\thesection}{Appendix \oldthesection}
\section{#1}\let\thesection\oldthesection}

\newtheorem{definition}[theorem]{Definition}

\newtheorem{example}[theorem]{Example}

\makeatletter
\providecommand{\leftsquigarrow}{%
  \mathrel{\mathpalette\reflect@squig\relax}%
}
\newcommand{\reflect@squig}[2]{%
  \reflectbox{$\m@th#1\rightsquigarrow$}%
}
\makeatother

\usepackage{scalerel,stackengine}
\stackMath
\newcommand\reallywidehat[1]{%
\savestack{\tmpbox}{\stretchto{%
  \scaleto{%
    \scalerel*[\widthof{\ensuremath{#1}}]{\kern-.6pt\bigwedge\kern-.6pt}%
    {\rule[-\textheight/2]{1ex}{\textheight}}%WIDTH-LIMITED BIG WEDGE
  }{\textheight}% 
}{0.5ex}}%
\stackon[1pt]{#1}{\tmpbox}%
}

\usepackage{todonotes}

\title{Classification of Wormhole Singularities}
\author{Jaime Negrete}
\address{Department of Mathematics, University of Georgia, Athens, USA}
\email{jaime.negrete@uga.edu}

\date{\today}

\begin{document}
\begin{abstract}
    We classify all wormhole singularities, i.e. cyclic quotient surface singularities admitting at least two extremal P-resolutions, thereby solving an open problem posed by Urzúa in his recent book \cite{LibroGian}.
    Our approach introduces a new combinatorial framework based on what we call the coherent graph of a framed triangulated polygon. As an application, we give an alternative proof of the Hacking-Tevelev-Urzúa theorem on the maximum number of extremal P-resolutions of a cyclic quotient singularity \cite{Flipping}*{Theorem 4.3}. 
\end{abstract}
\maketitle

{\hypersetup{linkcolor=black}
\tableofcontents
}

\vspace*{-3em}
\section{Introduction}
Singular varieties play a central role in algebraic geometry. They arise naturally from a wide range of constructions, including branched covers, finite group quotients, degenerations, and the higher-dimensional minimal model program. In the theory of singular surfaces, cyclic quotient singularities occupy a distinguished position: they form a class of singularities that frequently appear in degenerations and admit a rich combinatorial structure reflected both in their versal deformation spaces and in their minimal resolutions. Given coprime integers $0<q<m$, we define the (c.q.s) cyclic quotient singularity $\frac{1}{m}(1,q)$ as the germ at the origin of the quotient of $\mathbb{C}^2$ by the action $(x,y)\mapsto (\zeta x,\zeta^{q}y)$, where $\zeta$ is a primitive $m$-th root of unity.

Kollár and Shepherd-Barron \cite{KSB} introduced the notion of P-resolutions to study the deformation space of cyclic quotient singularities. Fixing a cyclic quotient singularity $(Q\in Y)$, they proved that there is a one-to-one correspondence between the components of the deformation space of $(Q \in Y)$ and the P-resolutions associated with $(Q \in Y)$. A P-resolution of $(Q\in Y)$ is a partial resolution $f\colon X \to (Q\in Y)$ such that $X$ has only T-singularities, and $K_X$ is ample relative to $f$. 
T-singularities are the 2-dimensional quotient singularities which admit a  $\mathbb{Q}$-Gorenstein one-parameter smoothing, i.e. they are either Du Val singularities or cyclic quotient singularities of type $\frac{1}{dn^2}(1,dna-1)$ with $0<a<n$, $d\geq 1$ and $\gcd(n,a)=1$. Among the non-Du Val T-singularities, those with $d=1$ are arguably the most relevant from several points of view. They are called Wahl singularities, as Wahl \cite{Wahl} was the first who studied topological and analytic invariants of smoothings of these singularities.
 
The cyclic quotient singularity $\frac{1}{m}(1,q)$ has a minimal resolution that replaces the singular point by a chain of smooth curves $E_i$ such that $E_i\simeq \mathbb{P}^1$ and $E_{i}^2=-e_i\leq -2$. The integers $e_i$ are exactly the numbers in the Hirzebruch-Jung (H-J) continued fraction 
$$\frac{m}{q}=e_1-\frac{1}{e_2-\frac{1}{\ddots \hspace{0,2em} -\frac{1}{e_r}}}:=[e_1,\cdots,e_r].$$
The chain $\frac{m}{q}=[e_1,\cdots,e_r]$ has a dual chain, which by definition is the Hirzebruch-Jung continued fraction $\frac{m}{m-q}=[k_1,\cdots,k_s]$ where $k_i\geq 2$.

Christophersen \cite{Christophersen} and Stevens \cite{Stevens} gave a combinatorial way to understand all P-resolutions of cyclic quotient singularities via zero continued fractions bounded by its dual Hirzebruch-Jung continued fraction (see also \cite{Flipping}). A zero continued fraction is a well-defined H-J continued fraction $[b_1,\cdots,b_s]$ for some $b_i\geq 1$ such that its value as a fraction is $0$.
We say that the fraction $[b_1,\cdots,b_s]$ is a zero continued fraction bounded by $\frac{m}{m-q}=[k_1,\cdots,k_s]$ if $[b_1,\cdots,b_s]$ is a zero continued fraction and $b_{i}\leq k_i$ for all $1\leq i\leq s$. 
There are the following well-known one-to-one correspondences 
$$\left\{\begin{array}{l}
    \text{Irreducible components} \\
    \hspace{1.4em} \text{ of } Def\left(\frac{1}{m}(1,q)\right)
  \end{array}\right\} \xleftrightarrow[]{1-1}\left\{\begin{array}{l}
    \text{P-resolutions} \\
    \hspace{0,6em} \text{ of } \frac{1}{m}(1,q)
  \end{array}\right\} \xleftrightarrow[]{1-1}\left\{\begin{array}{l}
    \text{zero continued fractions} \\
    \hspace{1em} \text{ bounded by } \frac{m}{m-q}
  \end{array}\right\}.
$$
Motivated by the work of Kollár, Mori, and Prokhorov \cite{KM}, \cite{MP} on extremal neighborhoods, Hacking, Tevelev, and Urzúa \cite{Flipping} defined extremal P-resolutions as a particular case of P-resolutions. An extremal P-resolution of $(Q\in Y)$ is a partial resolution $f\colon (C\subset X)\to (Q\in Y)$, such that $X$ has only Wahl singularities, there is one exceptional curve $C$ isomorphic to $\mathbb{P}^1$, and $K_X$ is ample relative to $f$. They proved that for any pair of coprime integers $0<q<m$ the c.q.s $\frac{1}{m}(1,q)$ can admit at most two distinct extremal P-resolutions \cite{Flipping}*{Theorem 4.3}.  

In view of the first correspondence \cite{KSB}, Urzúa and Vilches introduced wormhole singularities \cite{Wormhole} to study a particular phenomenon in the  Kollár--Shepherd-Barron--Alexeev (KSBA) compactification of the moduli space of surfaces of general type. A wormhole singularity is a cyclic quotient singularity which admits at least two different extremal P-resolutions. 
The second correspondence \cite{Christophersen},\cite{Stevens} together with \cite{Flipping}*{\textsection 4} lead to the notion of WW-decomposition of the H-J continued fraction of $\frac{m}{m-q}$ (Definition \ref{WW sequence}), 
and the existence of extremal P-resolutions associated with the c.q.s $\frac{1}{m}(1,q)$ is equivalent to the existence of some specific zero continued fractions bounded by $\frac{m}{m-q}$ with exactly two marks.

It is well-known that zero continued fractions of length $s$ are in bijection with triangulated $(s+1)$-gons together with a hidden index
\cite{Christophersen},\cite{Stevens},\cite{Flipping} (see also the appendix in \cite{T-Horikawas}). 
Recall that given a triangulated $(s+1)$-gon with vertices $P_0,\cdots,P_s$, one defines the index of a vertex $P_i$ as $$v_i:=(\text{number of diagonals}\\ \text{  from the vertex } P_i)+1.$$ To use this bijection systematically we say that a triangulated polygon is a framed triangulated polygon (Definition \ref{definition of framed triangulation}) if we choose a hidden index.

Using this bijection, there is a natural way to define a number associated with a WW-sequence, called the WW-index (Definition \ref{WW-index}).
The dual chain of a wormhole singularity defines a WW-sequence (Definition \ref{WW sequence} and \ref{definition of  wormhole singularity}). We say that a  wormhole singularity is a basic wormhole singularity if the WW-index of its dual fraction is greater than $1$. The classification of wormhole singularities can be reduced to the classification of basic wormhole singularities via the HTU algorithm in Section 3.

The classification of wormhole singularities has been open. A better understanding of wormhole singularities has been desired to potentially address the wormhole conjecture \cite{Wormhole} or to explore other interesting birational connections such as the relation between wormhole singularities and the birational geometry of Markov numbers \cite{Markov}.
From a topological point of view, wormhole singularities are related to the problem of fillings of lens spaces \cite{JonnyEvans}*{Open problem J.2}.
Basic wormhole singularities appear in Urzua's recent book \cite{Libro Gian}*{\textsection I} under the name of Uroburos, and its classification was posed as an open problem. 
In this paper we solve the open problem \cite{LibroGian}*{Open problem 8} through a careful study of the geometry of the triangulations associated with basic  wormhole singularities. We call these triangulations basic wormhole triangulations. In these triangulations, vertices of index greater than 2 play a crucial role; we refer to them as weights. To study basic wormhole triangulations, we consider the following approach.  
\begin{enumerate}
    \item Using the notion of WW-sequences we classify all candidates to basic wormhole triangulations, which we call accordion triangulations (Definition \ref{Definition of accordion triangulation}). They are the triangulated polygons with exactly $2$ vertices of index $1$.
    \item We introduce a graph associated with framed triangulated polygons called the coherent graph (Definition \ref{coherent graph}). It partially encodes the geometry of the diagonals of framed triangulated polygons, providing a simplified representation of framed triangulated polygons where vertices of index $1$ and $2$ are treated identically, while emphasizing the significance of weights. Coherent graphs provide us with the right framework for studying basic wormhole triangulations, since the coherent graphs of our candidates in $(1)$ are naturally equipped (after specifying a frame) with an explicit system of linear relations derived from the arrangement of the diagonals in the triangulation (Lemma \ref{linear relations in the coherent graph}). 
    \item To use the properties of coherent graphs, we must address the fact that accordion triangulations do not come with a frame. Among all the possible frames for an accordion triangulation, we consider an almost canonical frame which we call a standard frame (Definition \ref{standard framed}).
    That is, a frame where the hidden index is a weight adjacent to one of the vertices of index $1$ and the first entry of the associated zero fraction is equal to $1$. 
\end{enumerate}

The main results of this paper are the following:

\begin{theorem}\label{Intro: Main theorem}(= Theorem \ref{Main theorem})
Let $n\geq 2$ be a integer, and let $\mathscr{P}$ be a framed accordion triangulation with a standard frame
and weights $x_1,\cdots,x_n$. Then $\mathscr{P}$ is a basic wormhole triangulation with a standard frame if and only if there exists an integer $1\leq m\leq n-1$ such that the system $S_0\cup S_m$ of $2n$ linear equations
\begin{align*}
    S_0:y_i & =x_{n-i \text{(mod n)}}-k_{i}^{(0)}  \hspace{4em} \text{for } 1\leq i\leq n,\\
    S_m:y_i & =x_{(n-i)+m \text{ (mod n)}}-k_i^{(m)} \hspace{1.8em} \text{for } 1\leq i\leq n,
\end{align*}
is consistent; where the numbers $k^{(0)}_{i}$ and $k_{i}^{(m)}$ are equal to $3$ for all $i$, except at four specific positions determined by $n$ and $m$, where they are equal to $1$.
\end{theorem}
The consistency of this system of linear equations is described in Theorem \ref{Main theorem}. Moreover, when the system $S_0\cup S_m$ is consistent, we give a parametric solution depending on $\gcd(n,m)$-parameters and explain how to recover a parametric family of basic wormhole triangulations with a standard frame from this parametric solution.
\begin{corollary} \label{Intro: Basic wormhole triangulations algorithm}(=Corollary \ref{Basic wormhole triangulations algorithm})
We give a 3-step constructive algorithm to obtain all basic wormhole triangulations with a fixed number of weights.\\
    \textbf{Input}: An integer $n\geq 2$.\\
    \textbf{Output}: All parametric families of basic wormhole triangulations with $n$ weights.
\end{corollary}
This algorithm follows from a formal argument on how we can change the frame from the framed triangulated polygons arising in Theorem \ref{Intro: Main theorem}. Using this algorithm together with the HTU algorithm, one can explicitly describe the Hirzebruch-Jung continued fraction of all wormhole singularities. As an application of the ideas introduced in this paper we give an alternative proof of the Hacking-Tevelev-Urzúa theorem on the maximal number of extremal P-resolutions, see \cite{Flipping}*{Thm 4.3}.
\begin{theorem}\label{Intro: Alternative proof of HTU theorem on maximal number of extremal P-resolutions} (=Theorem \ref{Alternative proof of HTU theorem on maximal number of extremal P-resolutions})
    A cyclic quotient singularity $\frac{1}{m}(1,q)$ can admit at most two distinct extremal P-resolutions.
\end{theorem}

\subsubsection*{Summary of contents} In Section 2, we recall basic definitions and relevant results on Hirzebruch-Jung continued fractions, extremal P-resolutions, wormhole singularities, WW-sequences and triangulated polygons. Here, we recall the bijection between zero continued fractions of length $s$ and triangulated $(s+1)$-gon together with a hidden index.
In Section $3$, we recall the HTU algorithm to reduce the classification of wormhole singularities to the classification of basic wormhole singularities. We introduce the definitions of basic wormhole triangulation, accordion triangulation, companions of a basic wormhole triangulation, standard family of accordion triangulations, and coherent rotation of diagonals. Here, we introduce the crucial notion of coherent graph of a framed triangulated polygon. A key result is Lemma \ref{coherent graph of basic wormholes are the same} which enable us to study companions of a basic wormhole triangulation via some specific linear system of equations. This section contains the proofs of Theorem \ref{Intro: Main theorem}, Corollary \ref{Intro: Basic wormhole triangulations algorithm}, and Theorem \ref{Intro: Alternative proof of HTU theorem on maximal number of extremal P-resolutions}. 
In Section $4$, we classify all accordion triangulations with at most $5$ weights that admit a frame that converts them into a basic wormhole triangulation. In addition, we present an example of how to perform the last step in Corollary \ref{Intro: Basic wormhole triangulations algorithm} and how to obtain the wormhole singularity associated with those framed triangulated polygons.
\subsubsection*{\textbf{Acknowledgements}}
I thank my advisor, Valery Alexeev, for the discussions on a preliminary version of this document and for suggesting the terminology "accordion triangulations." I also thank Vicente Monreal, Joaqu\text{í}n Moraga, Giancarlo Urz\text{ú}a, and Juan Pablo Zuñiga for their comments. I gratefully acknowledge the University of Georgia and the Department of Mathematics at UGA for awarding me the
Franklin College Research Assistantship for the 2024-2025 academic year.

\section{Preliminaries on extremal P-resolutions and  wormhole singularities}

We begin by setting notation and reviewing fundamental results about Hirzebruch-Jung continued fractions.

\begin{definition}\cite{jp-gian}*{Definition 2.1}
    A collection $\{e_1,\cdots,e_r\}$ of positive integers admits a Hirzebruch-Jung (H-J) continued fraction 
    $$[e_1,\cdots,e_r]:=e_1-\frac{1}{e_2-\frac{1}{\ddots \hspace{0,2em} -\frac{1}{e_r}}},$$
    if $[e_i,\cdots,e_{r}]>0$ for all $i\geq 2$, and $[e_1,\cdots,e_r]\geq 0$. Its value is the rational number $[e_1,\cdots,e_r]$, and its length is $r$. The fraction $[e_1,\cdots,e_r]$ is also called a chain.
\end{definition}

If the collection $\{e_1,\cdots,e_r\}$ satisfies that $e_{i}\geq 2$ for all $1\leq i\leq r$, then the value of the chain $[e_1,\cdots,e_r]$ is a rational number $\frac{m}{q}$ greater than $1$. Conversely, any fraction $\frac{m}{q}>1$ has a unique H-J continued fraction with entries greater than $1$. This gives a one-to-one correspondence between H-J continued fractions with entries greater than $1$ and $\mathbb{Q}_{>1}$.
A zero continued fraction is a Hirzebruch-Jung continued fraction whose value as a fraction is equal to $0$. Every chain $\frac{m}{q}=[e_1,\cdots,e_r]$ has a dual chain, which by definition is the Hirzebruch-Jung continued fraction $\frac{m}{m-q}=[k_1,\cdots,k_s]$ where $k_i\geq 2$. We have the relation
$$[e_1,\cdots,e_r,1,k_s,\cdots,k_1]=0.$$

\remark \label{dualizing using dot diagram} From the Riemenschneider's dot diagram \cite{Dot diagram}, if we write 
$$\frac{m}{q}=[\underbrace{2,\cdots,2}_{a_1},b_1,\underbrace{2,\cdots,2}_{a_2},b_2,\cdots,\underbrace{2,\cdots,2}_{a_{c-1}},b_{c-1},\underbrace{2,\cdots,2}_{a_c}],$$
where $a_i\geq 0$ and $b_i\geq 3$ for all $i$, then
$$\frac{m}{m-q}=[a_1+2,\underbrace{2\cdots,2}_{b_1-3},a_2+3,\underbrace{2,\cdots,2}_{b_2-3},a_3+3,\cdots,a_{c-1}+3,\underbrace{2,\cdots,2}_{b_{c-1}-3},a_{c}+2].$$
\begin{proposition} (\cite{KSB}*{Proposition 3.10}) 
    Wahl singularities are cyclic quotient singularities of the form $\frac{1}{n^2}(1,na-1)$ with $0<a<n$ and $\gcd(n,a)=1$.
\end{proposition}
Wahl singularities are minimally resolved by W-chains. By an important result of Wahl we have an algorithm to recognize Wahl singularities from its minimal resolution.

\begin{proposition}\label{R,L-operations}\cite{KSB}*{ Proposition 3.11} (W-chains algorithm). For any Wahl singularity $\frac{1}{n^2}(1,na-1)$ we have:
\begin{enumerate}
    \item[(i)]  If $n=2$, then the W-chain is $[4]$.
    \item[(ii)] If $[e_1,\cdots,e_r]$ is a W-chain, then so are $[2,e_1,\cdots,e_r+1]$ and $[e_1+1,e_2,\cdots,e_r,2]$.
    \item[(iii)] Every W-chain can be obtained by starting with the chain in $(i)$ and iterating the steps described in $(ii)$.
\end{enumerate}   
\end{proposition}

\remark The operations in $(ii)$ are called $\mathcal{L}$-operation and $\mathcal{R}$-operation, respectively. There is an alternative description of this algorithm via duals chain of W-chains. If $[k_1,\cdots,k_s]$ is a dual W-chain, then so are $[2,k_1,\cdots,k_s+1]$ and $[k_1+1,k_2,\cdots,k_s,2]$. In such a case, these operations are called $\mathcal{R}^{\vee}$-operation and $\mathcal{L}^{\vee}$-operation, respectively. 

\begin{definition}\cite{Flipping}*{\textsection 4}
Let $0<q<m$ be coprime integers, and let ($Q \in Y$) be a cyclic quotient singularity $\frac{1}{m}(1,q)$. An extremal P-resolution of $(Q\in Y)$ is a partial resolution $f\colon (C\subset X)\to (Q\in Y)$, such that $X$ has only Wahl singularities, there is one exceptional curve $C$ and isomorphic to $\mathbb{P}^1$, and $K_X$ is ample relative to $f$.
\end{definition}

\begin{theorem}\label{Theorem 4.3 in Flipping paper}\cite{Flipping}*{Theorem 4.3}
    A cyclic quotient singularity $\frac{1}{m}(1,q)$ can admit at most
two distinct extremal P-resolutions.
\end{theorem}

Given coprime integers $0<q<m$, we can find all extremal P-resolutions of the cyclic quotient singularity $\frac{1}{m}(1,q)$ by looking at the dual chain $\frac{m}{m-q}$.
\begin{proposition}\cite{Flipping}*{\textsection 4} \label{Bijection between extremal P-resolutions and pairs of zero continued fractions with two marks}
Let $0<q<m$ be coprime integers. If $\frac{m}{m-q}=[k_1,\cdots,k_s]$ where $k_{i}\geq 2$, then there is a bijection between extremal P-resolutions associated with the cyclic quotient singularity $\frac{1}{m}(1,q)$ and pairs $1\leq \alpha<\beta\leq s$ such that
$$[k_1,\cdots,k_{\alpha-1},k_{\alpha}-1,k_{\alpha+1},\cdots,k_{\beta-1},k_{\beta}-1,k_{\beta+1},\cdots,k_s]=0.$$
\end{proposition}

\begin{definition}\cite{Flipping}*{Definition 4.5}\label{WW sequence} A sequence $\{b_1,\cdots,b_s\}$, $b_i>1$ is a WW-sequence if there exists $1\leq \alpha<\beta\leq s$ such that
$$[b_1,\cdots,b_\alpha-1,\cdots,b_\beta-1,\cdots,b_s]=0.$$
The numbers $\alpha$ and $\beta$ are called the indices of the WW-sequence, and we say that the zero continued fraction $0=[b_1,\cdots,b_{\alpha}-1,\cdots, b_{\beta}-1,\cdots,b_s]$ is a WW-decomposition of the Hirzebruch-Jung continued fraction $[b_1,\cdots,b_s]$.
\end{definition}
\remark We add the notion of WW-decomposition to the original definition of WW-sequence.

\begin{definition}\cite{Wormhole}*{Definition 2.7}\label{definition of  wormhole singularity}
    A  wormhole singularity is a cyclic quotient singularity $\frac{1}{m}(1,q)$ with $0<q<m$ coprime integers, which admits at least two distinct extremal P-resolutions. Equivalently, the Hirzebruch-Jung continued fraction of $\frac{m}{m-q}$ has at least two different WW-decompositions.
\end{definition}

The set of WW-decompositions of a given fraction can be computed explicitly due to the relation between zero continued fractions and triangulated polygons, see \cite{Christophersen},\cite{Stevens},\cite{Flipping},\cite{T-Horikawas}. If $s\geq 2$, a triangulation of $s+1$ sides corresponds to drawing some diagonals over a convex polygon of $s+1$ sides (with vertices $P_0,\cdots,P_s$, ordered counterclockwise) such that: diagonals do not intersect each other (except maybe over the vertices), and the polygon is divided into triangles.
Given a triangulation, we define the index of a vertex $P_i$ as $$v_i:=(\text{number of diagonals from the vertex } P_i)+1.$$ 
The vector $(v_0,v_1,\cdots,v_n)$ is called the vector of indices of the triangulated polygon $P_0,\cdots,P_n$.

\begin{theorem}\label{1-dimensional zero fractions and triangulations}\cite{Wormhole}*{\textsection 2}
    Let $(b_1,\cdots,b_s)$ be a sequence of positive integers. Then $[b_1,\cdots,b_s]=0$ if and only if there exists a positive integer $b_0$ such that $(b_0,b_{1},\cdots,b_{s})$ is the vector of indices of some triangulated polygon of $s+1$ sides.
\end{theorem}
\begin{lemma} \label{trivial}
    Let $\mathcal{P}$ be a convex polygon with $s+1$ sides. The indices from the vertices of $\mathscr{P}$ will be taken $\text{mod } s+1$. Consider $[b_1,\cdots,b_s]=0$ for a triangulation of $\mathcal{P}$, then
    \begin{enumerate}
        \item $b_0+b_1\cdots+b_s=3(s-1)$ where $b_0$ is the positive integer from Theorem \ref{1-dimensional zero fractions and triangulations}.
        \item At least two $b_i$ must be equal to $1$. Furthermore, for $s\geq 3$, the entries equal to $1$ cannot be in consecutive positions.
        \item Let $s\geq 2$. If $[b_1,\cdots,b_{i-1},1,b_{i+1},\cdots,b_s]=0$ for $i\neq 1,s$, then $0=[b_1,\cdots,b_{i-1}-1,b_{i+1}-1,\cdots,b_s]$. If $[1,b_2,\cdots,b_s]=0$, then $[b_2-1,\cdots,b_s]=0$; if $[b_1,\cdots,b_{s-1},1]=0$, then $[b_1,\cdots,b_{s-1}-1]=0$. Furthermore, if $s\geq 4$ and $[b_1,\cdots,b_{i-1},1,b_{i+1},\cdots,b_s]=0$ for $i\neq 1,s$, then $b_{i-1}$ and $b_{i+1}$ cannot be both equal to $2$.
        \item Let $s\geq 3$. Then $b_i\leq s-1$ for any $0\leq i\leq s$. Moreover, if $b_i=s-1$, then $b_{i-1}=b_{i+1}=1$ and $b_{j}=2$ for any $j\neq i-1,i,i+1$.
    \end{enumerate}
\end{lemma}
\begin{proof}
    These properties are straightforward.
\end{proof}

\begin{remark}
    The family of triangulated $(s+1)$-gons of Lemma \ref{trivial}(4) (for all $s\geq 3$) is called the trivial family.
\end{remark}

The number $b_0$ in Theorem \ref{1-dimensional zero fractions and triangulations} is called the hidden index of the triangulation associated with the zero continued fraction $[b_1,\cdots,b_s]=0$. By Lemma \ref{trivial}(1), the number $b_0$ is completely determined by the linear relation
$$b_0:=3(s-1)-\sum_{i=1}^{s}b_i.$$
Therefore, we can write the zero continued fraction $[b_1,\cdots,b_s]$ as $[b_1,\cdots,b_s\mid b_0]$. The latter notation is called extended zero chain.

\remark \label{cyclically rotate entries of an extended zero chain corresponds to choose another hidden index} Let $[b_1,b_2,\cdots,b_{s}\mid b_0]$ be an extended zero chain. By Theorem \ref{1-dimensional zero fractions and triangulations}, there exists a triangulated $(s+1)$-gon $\mathscr{P}$ with vector of indices $(b_0,b_1,\cdots,b_s)$, where $b_0$ is the hidden index. By relabeling the vertices we can consider the hidden index as the index corresponding $b_i$. It corresponds to a cyclic permutation of the entries of the extended zero chain $[b_0,b_1,\cdots,b_s\mid b_0]$ in such a way that $b_i$ is the new hidden index of the extended zero chain. Therefore, a triangulated $(s+1)$-gon encodes $(s+1)$ extended zero chains, one for each choice of a hidden index.

\begin{lemma}\label{v_0 is well-defined for a wormhole.}
    Let $\{k_1,\cdots,k_s\}$ be a WW-sequence. Assume that $0=[b_1,\cdots,b_s]$ and $0=[b'_1,\cdots,b'_s]$ are WW-decompositions of $[k_1,\cdots,k_s]$. If $[b_1,\cdots,b_s \hspace{0,2em} \mid \hspace{0,2em}v_0]$ and $[b'_1,\cdots,b'_s \hspace{0,2em} \mid \hspace{0,2em}v_0']$ are the corresponding extended zero chains, then $v_0=v_0'$.
\end{lemma}
\begin{proof}
    By Definition \ref{WW sequence}, the entries of these zero fractions were obtained by subtracting 1 from $[k_1,\cdots,k_s]$ in exactly two different positions, so
    \begin{align}\label{v_0=v_0', subtract exactly 2}
    2+\sum_{i=1}^{s}b_i=\sum_{i=1}^{s}k_i=2+\sum_{i=1}^{s}b'_i.
    \end{align}
    By Lemma \ref{trivial}(1), we obtain the relations
    \begin{align}\label{v_0=v_0', both set of indices are equal to a constant number}
        v_0+\sum_{i=1}^{s}b_{i}=3(s-1)=v_0'+\sum_{i=1}^{s}b'_i.
    \end{align}
    The relations (\ref{v_0=v_0', subtract exactly 2}) and (\ref{v_0=v_0', both set of indices are equal to a constant number}) give the desired conclusion.
\end{proof}

\begin{definition}\label{WW-index}
    The WW-index of a WW-sequence is the hidden index of any of its WW-decompositions.
\end{definition}

\begin{definition} \label{definition of basic wormhole singularity}
    Let $0<q<m$ be coprime integers, and let $\frac{m}{m-q}=[k_1,\cdots,k_s]$ where $k_i\geq 2$. The cyclic quotient singularity $\frac{1}{m}(1,q)$ is a basic  wormhole singularity if $\frac{1}{m}(1,q)$ is a wormhole singularity and the WW-index of the WW-sequence $\{k_1,\cdots,k_s\}$ is greater than $1$.
\end{definition}

\begin{definition}\label{m-cyclic permutation}
    Let $\{b_1,\cdots,b_s\}$ be a WW-sequence, and let $b_0$ be its WW-index. Denote by $P\colon \mathbb{Z}^{s+1}\to \mathbb{Z}^{s+1}$ the cyclic permutation map $(x_1,\cdots,x_{s},x_{s+1})\mapsto (x_{s+1},x_1,\cdots,x_{s})$. For $r\geq 0$, the $r$-cyclic permutation of $\{b_1,\cdots,b_s\}$ is the sequence obtained by dropping the last entry of the vector $P^r(b_1,\cdots,b_s,b_0)$. The $0$-cyclic permutation of $\{b_1,\cdots,b_s\}$ is simply $\{b_1,\cdots,b_s\}$.
\end{definition}

\begin{lemma}\label{bound for the number of WW-decompositions under cyclic permutation}
    Let $\{b_1,\cdots,b_s\}$ be a WW-sequence with $k$ WW-decompositions. If $b_{s-(r-1)}\geq 3$, then the $r$-cyclic permutation of $\{b_1,\cdots,b_s\}$ is a WW-sequence with at least $k$ WW-decompositions.
\end{lemma}
\begin{proof}
    Each WW-decomposition of $\{b_1,\cdots,b_s\}$ can be expressed as an extended zero chain. By cyclic permutation of the numbers in these zero chains, we obtain other extended zeros chains (Remark \ref{cyclically rotate entries of an extended zero chain corresponds to choose another hidden index}). The condition $b_{s-{(r-1)}}\geq 3$ ensures that the corresponding zero chains are WW-decompositions of the $r$-cyclic permutation of $\{b_1,\cdots,b_s\}$.
\end{proof}

\section{Classification of  wormhole singularities}

In this section we classify  wormhole singularities. We consider the following approach:
\begin{enumerate}
    \item Reduce the classification of wormhole singularities to the classification of basic wormhole singularities via the HTU algorithm (Lemma \ref{HTU algorithm}). 
    \item Define basic wormhole triangulations (Definition \ref{definition of basic wormhole triangulation}) and note that we can classify basic wormhole singularities by classifying basic wormhole triangulations (Remark \ref{It is enough to classify basic womrhole triangulations}).
    \item Classify candidates to basic wormhole triangulations: accordion triangulations (Definition \ref{Definition of accordion triangulation}).
    \item Introduce the coherent graph of a framed triangulated polygon (Definition \ref{coherent graph}) and its properties.
    \item Define a standard frame for accordion triangulations (Definition \ref{standard framed}). Then define standard family of accordion triangulations and coherent rotation of diagonals (Definition \ref{Standard family of accordion triangulation of weight n and coherent rotation of diagonals for elements of the standard family of weight n}).
    \item Classify framed accordion triangulations with a standard frame that are basic wormhole triangulations with a standard frame (Theorem \ref{Main theorem}).
    \item Basic wormhole triangulations algorithm (Corollary \ref{Basic wormhole triangulations algorithm}).
\end{enumerate}

Let $[b_1,\cdots,b_s]$ be the dual chain of a wormhole singularity such that the WW-index (Definition \ref{WW-index}) of the WW-sequence $\{b_1,\cdots,b_s\}$ is equal to 1. By Theorem \ref{1-dimensional zero fractions and triangulations}, each WW-decomposition of $[b_1,\cdots,b_s]$ corresponds to a triangulated $(s+1)$-gon with hidden index equal to $1$. In each WW-decomposition, by Lemma \ref{trivial}(2), one of the entries must be equal to $1$. 
Furthermore, we must have $b_1,b_s\geq 2$ unless the polygon is a triangle (i.e. $s=2$). If $s>2$, by Definition \ref{WW sequence}, either $b_1=2$ or $b_s=2$. Let $P$ be the vertex $P_1$ or $P_s$ that does not correspond to that entry, let $\Delta$ be the triangle $P_{s},P_0$,$P_1$, and let $\ell$ be the diagonal connecting $P_1$ and $P_{s}$.
Removing $(\Delta\smallsetminus \ell)$ from the triangulated $(s+1)$-gon produces a triangulated $s$-gon, and we choose its hidden index as the vertex $P$ in the new polygon. By Lemma \ref{trivial} (3), it gives a new zero fraction of length $s-1$. 
We do the same for every WW-decomposition of $[b_1,\cdots,b_s]$ and
continue with this algorithm until we obtain a WW-sequence with WW-index $v_0>1$. We call this algorithm the HTU algorithm because Hacking-Tevelev-Urzúa used this idea in \cite{Flipping}*{\textsection 4} to prove some facts about wormhole singularities by reducing the analysis to the case of basic wormhole singularities (Definition \ref{definition of basic wormhole singularity}). There is a degenerate case of this algorithm, which is a missing case in the argument in \cite{Flipping},\cite{LibroGian}, and \cite{Wormhole}, where this algorithm does not finish in a WW-sequence with WW-index $v_0>1$. We prove the following:

\begin{lemma}\label{HTU algorithm}
    Let $\{b_1,\cdots,b_s\}$ be a WW-sequence with at least two different WW-decompositions. If the WW-index of $\{b_1,\cdots,b_s\}$ is equal to $1$, then the HTU algorithm produces either a WW-sequence with WW-index greater than $1$, or the WW-sequence $\{3,\cdots,3,2,2,3,\cdots,3\}$ with WW-index equal to $1$. In both cases, the output of the HTU algorithm is a WW-sequence that has the same number of WW-decompositions as $[b_1,\cdots,b_s]$
\end{lemma}
\begin{proof}
    It follows immediately from the definition of the HTU algorithm. The choice of frame after removing the corresponding triangle ensures that we keep the number of WW-decompositions constant throughout this algorithm. The degenerate case $[3,\cdots,3,2,2,3,\cdots,3]$ is a missing case in the argument in \cite{Flipping},\cite{LibroGian}, and \cite{Wormhole}. It comes from intermediate WW-sequences $\{b_1',\cdots,b'_{s_k}\}$ with hidden index equal $1$ that we produce while running the HTU algorithm. Specifically, two of its WW-decompositions have indices $1<\beta_1$ and $\alpha_2<s_k$. Since the hidden index is $1$, then $b'_{1}=3$ and $b'_{s_k}=3$.
    By Lemma \ref{trivial}(3), they descend to $0=[b'_1-1,b'_2,\cdots,b'_{s_k}]$ and $0=[b'_1,\cdots,b'_{s_k-1},b'_{s_k}-1]$. By \cite{LibroGian}*{Prop. 1.18}, the WW-sequences associated with each of these zero fractions are dual W-chains. 
    Since dual W-chains are constructed via $\mathcal{L}^{\vee}$-operations and $\mathcal{R}^{\vee}$-operations from the initial chain $[2,2,2]$ and these operations only change the ends of the fractions in each steps, it is straightforward to check that the only possibility is $[3,\cdots,3,2,2,3,\cdots,3]$.
\end{proof}
\remark\label{Degenerate case has 2 WW-decompositions} The WW-sequence $\{3,\cdots,3,2,2,3,\cdots,3\}$ has exactly two WW-decompositions.

\remark To classify wormhole singularities, it is sufficient to classify basic wormhole singularities and apply the HTU algorithm backwards.

\begin{definition}\label{definition of basic wormhole triangulation}
    A wormhole triangulation is a triangulation that represents any of the WW-decompositions of the dual fraction of a  wormhole singularity. A basic wormhole triangulation is a wormhole triangulation associated with a basic wormhole singularity.
    If $\mathscr{P}$ is a basic wormhole triangulation associated with a basic wormhole singularity, then any other triangulation associated with other WW-decomposition of the same basic wormhole singularity is called a companion of $\mathscr{P}$.
\end{definition}

\begin{definition}\label{Definition of accordion triangulation} 
    An accordion triangulation of $s+1$ sides with base at $b_0$ is a triangulated $(s+1)$-gon $\mathscr{P}$ with vertices $b_0,b_1,\cdots,b_s$ constructed recursively as follows:
    \begin{itemize}
        \item{\textbf{Step I}}: Let $\mathcal{P}$ be a convex polygon with vertices $b_0,b_1,\cdots,b_s$. Draw the diagonal $\ell_0$ from $b_0$ to $b_2$, and define $\Delta_0$ as the triangle with vertices $b_0, b_1$ and $b_2$. Set $J=1$ and continue to Step II.
        \item{\textbf{Step II}}: We have two options:
        \begin{itemize}
            \item If $1\leq J\leq (s+1)-4$: 
            Let $\mathscr{P}^{(J)}$ be the polygon obtained from removing $\bigcup_{k=0}^{J-1}\left(\Delta_{k}\smallsetminus \ell_{k}\right)$ from $\mathcal{P}$.
            Let $b'_0$ and $b'_{1}$ be the vertices in $\ell_{J-1}\cap \mathcal{P}$. Let $b'_{L}$ be the vertex adjacent to $b'_0$ in $\mathscr{P}^{(J)}$ that is not $b'_{1}$, and let $b'_{R}$ be the vertex adjacent  to $b'_1$ in $\mathscr{P}^{(J)}$ that is not $b'_{0}$.
            Choose the diagonal $\ell_{J}$ as either the diagonal from $b'_{0}$ to $b'_{R}$ or the diagonal from $b'_{1}$ to $b'_{L}$. Define $\Delta_J$ as the triangle in $\mathscr{P}^{(J)}$ determined by $\ell_J$, redefine $J$ as $J+1$, and return to Step II.
            \item If $J=(s+1)-3$: Continue to Step III. 
        \end{itemize} 
        \item{\textbf{Step III}}: Let $\mathscr{P}$ be the convex polygon $\mathcal{P}$ triangulated with the diagonals $\ell_0,\ell_1,\cdots,\ell_{(s+1)-4}$.
    \end{itemize}
    An accordion triangulation is an accordion triangulation of $s+1$ sides with base at $b_0$ for some $s\geq 3$ and some vertices $b_0,b_1,\cdots,b_{s}$.
\end{definition}

\remark \label{Every accordion triangulation of s+1 sides with s>3 has a weight adjacent to a 1.} By Lemma \ref{trivial}(3), every index one vertex of a triangulated $(s+1)$-gon with $s\geq 4$ is adjacent to at least one vertex of index greater than $2$. By Definition \ref{Definition of accordion triangulation}, if the triangulated $(s+1)$-gon $\mathscr{P}$ is an accordion triangulation with $(s+1)\geq 4$, then each index one vertex of $\mathscr{P}$ is adjacent to exactly one (not necessarily the same for both index one vertices) vertex of index greater than $2$.
Since every accordion triangulation is constructed from the data of a base vertex together with a quite particular set of instructions in how we draw the diagonals, this observation tells us that for $(s+1)\geq 4$ this data is equivalent to give the position of exactly one index 1 vertex, the position of all vertices of index greater than $2$ and the exact index of each of these vertices.

\begin{figure}[h]
    \includegraphics[width=10cm]{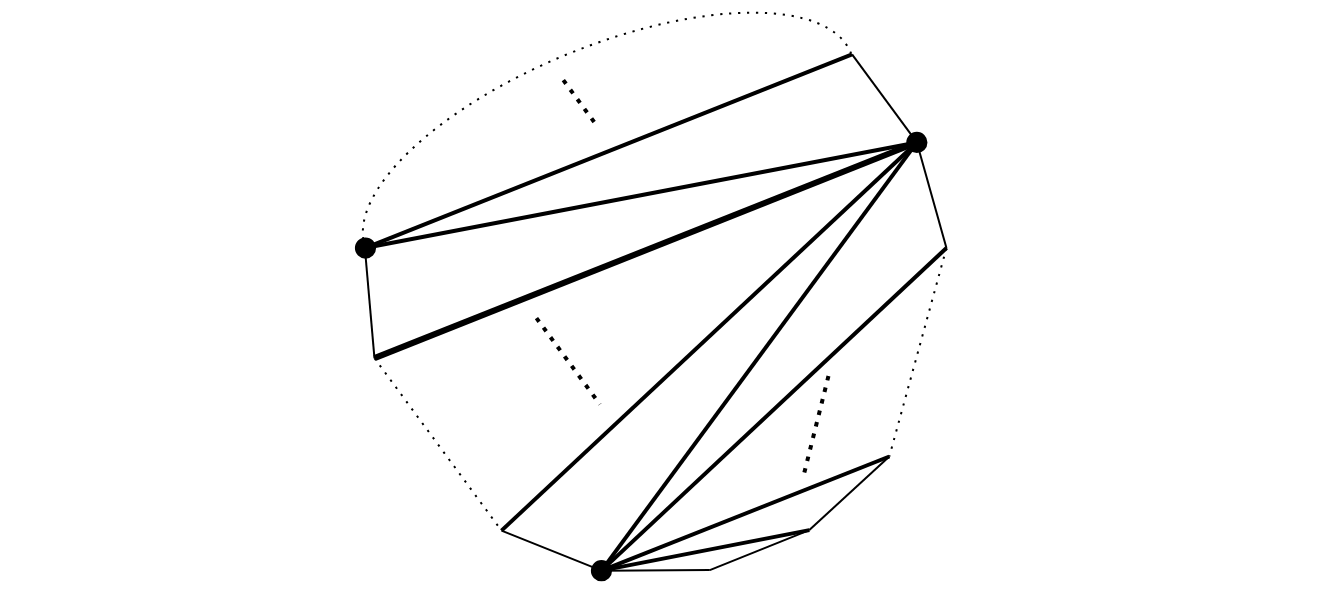}
    \caption{Construction of an accordion triangulation.}
\end{figure}

\begin{lemma}\label{every triangulation with exactly two 1's is an accordion triangulation}
    Let $\mathscr{P}$ be a triangulated polygon. Then $\mathscr{P}$ has exactly two index $1$ vertices if and only if $\mathscr{P}$ is an accordion triangulation.
\end{lemma}

\begin{proof}
    Let $\mathscr{P}$ be a triangulated polygon of $s+1$ sides with exactly two index 1 vertices. By definition, $\mathscr{P}$ is constructed from a convex polygon of $s+1$ sides $\mathcal{P}$ with vertices $b_0,b_1,\cdots,b_{s}$ by drawing some diagonals over $\mathcal{P}$. If $(s+1)=4$, the result is obvious. Assume that $(s+1)\geq 5$.
    By Lemma \ref{trivial}(2), every triangulated polygon must have at least two index $1$ vertices. Without loss of generality assume that $b_1$ is one of the index $1$ vertices of $\mathscr{P}$, i.e. there is a diagonal $\ell_0$ between $b_0$ and $b_2$. It determines a triangle $\Delta_0$ with vertices $b_0$, $b_1$, and $b_2$.
    Let $\mathscr{P}^{(1)}$ be the polygon obtained from removing $(\Delta_0\smallsetminus \ell_0)$ and denote by $b_0'$ and $b_1'$ the vertices in $\ell_0 \cap \mathscr{P}^{(1)}$.
    Since $\mathscr{P}^{(1)}$ is a polygon of $s$ sides that must be triangulated, by Lemma \ref{trivial}(2), we must have at least two index 1 vertices for such a triangulation of $\mathscr{P}^{(1)}$ and these index 1 vertices cannot be in consecutive positions. This implies that we must have at least one diagonal from $b_0'$ or from $b_1'$.
    If we do not have a diagonal as the diagonal $\ell_1$ in the Definition \ref{Definition of accordion triangulation}, then any diagonal $l_1'$ from $b_0'$ or $b_1'$ must divide the polygon $\mathscr{P}^{(1)}$ into two polygons of at least $4$ sides (counting $l'_1$ as one side for each).
    Each of these sub-polygons must be triangulated, and therefore, each must contribute at least one vertex of index $1$ to the triangulation $\mathscr{P}$ (Lemma \ref{trivial}). 
    Since $l'_1$ is not as $\ell_1$ in Definition \ref{Definition of accordion triangulation}, then gluing $(\Delta_0 \smallsetminus \ell_0)$ to the triangulation defined over $\mathscr{P}^{(1)}$ does not change any index one vertex of the triangulation over $\mathscr{P}^{(1)}$ and $(\Delta_0 \smallsetminus \ell_0)$ contributes with another index $1$ vertex for $\mathscr{P}$.
    Since the triangulation $\mathscr{P}$ is the same as the triangulation over $\mathscr{P}^{(1)}$ glued with $(\Delta_0 \smallsetminus \ell_0)$, it contradicts the fact that $\mathscr{P}$ has exactly two index $1$ vertices. Therefore, the diagonal $\ell_1$ must be as in Definition \ref{Definition of accordion triangulation}. The result follows recursively.
\end{proof}

\begin{definition} \label{definition of framed triangulation}
    A framed triangulated polygon $\mathscr{P}$ is a triangulated polygon together with the choice of a hidden index. That is, $\mathscr{P}$ is given by an extended zero fraction.
\end{definition}

\begin{proposition}
    \label{basic wormhole are accordion triangulations}
    Let $\mathscr{P}$ be a basic wormhole triangulation. Then $\mathscr{P}$ is a framed accordion triangulation.
\end{proposition}
\begin{proof}
    By Definitions \ref{definition of basic wormhole triangulation} and \ref{definition of basic wormhole singularity}, the triangulation $\mathscr{P}$ is an extended zero fraction $[b_1,\cdots,b_s\mid b_0]$ where $b_0>1$. By Definition \ref{definition of basic wormhole triangulation}, it represents a WW-decomposition of a WW-sequence $\{b_{1}',\cdots,b_{s}'\}$ for some $b'_{i}>1$. By Definition \ref{WW sequence}, WW-decompositions of $\{b_1',\cdots,b_s'\}$ are obtained by subtracting $1$ from $\{b'_1,\cdots,b_s'\}$ in two different positions $1\leq \alpha<\beta \leq s$. By Theorem \ref{1-dimensional zero fractions and triangulations}, the numbers $(b_1,\cdots,b_s,b_0)$ correspond to the vector of indices of some triangulated $(s+1)$-gon. By Lemma \ref{trivial}(2), at least two $b_i$ must be equal to $1$. Since $b_{0}>1$, we obtain $b_{\alpha}=1$, $b_{\beta}=1$ and $b_i=b'_{i}$ for all other $i$. That is, $b'_{\alpha}=2$, $b'_{\beta}=2$, and $b'_{i}=b_i$ for all other $i$. The result follows from Lemma \ref{every triangulation with exactly two 1's is an accordion triangulation} and Definition \ref{definition of framed triangulation}.
\end{proof}

\remark \label{It is enough to classify basic womrhole triangulations} Let $\mathscr{P}=[b_1,\cdots,b_s \mid b_0]$ be a basic wormhole triangulation. It corresponds to a WW-decomposition with indices $1\leq \alpha<\beta\leq s$ of a WW-sequence $\{b_{1}',\cdots,b_{s}'\}$ for some $b'_{i}>1$. As in the previous proof, we obtain $b_{\alpha}=1$, $b_{\beta}=1$ and $b_i=b'_{i}$ for all other $i$. That is, $b'_{\alpha}=2$, $b'_{\beta}=2$, and $b'_{i}=b_i$ for all other $i$. Since $[b'_1,\cdots,b'_s]$ is the dual fraction of the associated  wormhole singularity, we can explicitly recover the Hirzebruch-Jung fraction associated with the  wormhole singularity via the Riemenschneider's dot diagram (Remark \ref{dualizing using dot diagram}). Therefore, by classifying basic wormhole triangulations, we are classifying basic wormhole singularities.

\remark \label{No wormhole triangulations of weight $1$} Accordion triangulations do not come with a frame. Since basic wormhole triangulations are framed accordion triangulations, it is sufficient to classify accordion triangulations that admit a frame that makes them a basic wormhole triangulation. Note that there are accordion triangulations that do not have such a frame.
For instance, triangulations in the trivial family (see Lemma \ref{trivial}(4)) are accordion triangulations that have no frame that makes them basic wormhole triangulations. Indeed, by Lemma \ref{trivial}(1), elements in the trivial family are the only accordion triangulations with exactly one vertex of index greater than $2$. Therefore, there are no basic wormhole triangulations with exactly one vertex of index greater than $2$.

\begin{definition} \label{Two framed triangulations are equal} Let $\mathscr{P}=[b_1,\cdots,b_s\mid b_0]$ and $\widetilde{\mathscr{P}}=[b_{1}',\cdots,b'_s\mid b'_{0}]$ be framed triangulated polygons. We say that $\mathscr{P}$ and $\widetilde{\mathscr{P}}$ are equal if $b_i=b_{i}'$ for all $0\leq i\leq n$. In such a case, we use the notation $\mathscr{P}=\widetilde{\mathscr{P}}$.
\end{definition}

%%%%%%%%%%%%%%%%%%%%%%%%%% The reason why I will keep thr definition of basic wormhole triangulations associated to the same family of singularities is because I want to prove Giancarlos's fact. %%%%%%%%%%%%%%%%%%%%%%%%%%%%%%%%%

%%%%%%%%%%%%%%%%%%%%%%%%%%%%%%%%%%%%%%%%%%%%%%%%%%%%%%%%%%%%%%%%%%%%%%%%%%%%%%%%%%%%%%%%%%%% \begin{definition} Let $\mathscr{P}=[b_1,\cdots,b_s\mid b_0]$ and $\widetilde{\mathscr{P}}=[b_{1}',\cdots,b'_s\mid b'_{0}]$ be basic wormhole triangulations. We say that $\mathscr{P}$ and $\widetilde{\mathscr{P}}$ are basic wormhole triangulations associated to the same family of basic wormhole singularities if $b_{0}>1$, and $b_{i}=b_{i}'$ for all $0\leq i\leq s$, except at $4$ positions where we are swapping the pairs of $1's$ in each extended zero fraction with a pair of $2$'s in the other one.\end{definition}%%%%%%%%%%%%%%%%%%%%%%%%%%%%%%%%%%%%%%%%%%%%%%%%%%%%%%%%%%%%%%%%%%%%%%%%%%%%%%%

Now we introduce the key definition of coherent graph of a framed triangulated polygon.

\begin{definition}\label{coherent graph}
    Let $\mathscr{P}=[b_1,\cdots,b_s\mid b_0]$ be a framed triangulated $(s+1)$-gon. Reading the vector $v=(b_1,\cdots,b_s,b_0)$ from left to right, define $x_i$ as the 
    $i$-th entry that is greater than $2$, and $j_i$ as the position such that $b_{j_i}=x_i$.
    Let $n$ be the number of $x_i's$ associated with the vector $v$. For each $1\leq i \leq n-1$, define $y_i$ as $j_{i+1}-j_{i}-1$. If $n\geq 2$, we define $y_n$ as $(s+1)-n-\sum_{i=1}^{n-1}y_i$. If $n=1$, we define $y_n$ as $(s+1)-n$. Construct the coherent graph $G_{\mathscr{P}}$ as follows: 
    \begin{itemize}
        \item{\textbf{Vertices}}: $G_{\mathscr{P}}$ has $n$ vertices, one for each $x_i$. The vertex associated with $x_i$ is denoted by $x_i$.
        \item{\textbf{Edges}}: For $1\leq i\leq n$, there is an edge between $x_i$ and $x_{i+1}$, labeled with $y_i$. Additionally, there is an edge between $x_n$ and $x_1$, labeled with $y_n$.
    \end{itemize}
    The vertices of $\mathscr{P}$ corresponding to the $x_{i}'s$ are called the weights of $\mathscr{P}$. We denote the coherent graph of $\mathscr{P}$ by $G_{\mathscr{P}}=([x_1,\cdots,x_{n-1}\mid x_n],(y_1,\cdots,y_n))$. If only the weights of $\mathscr{P}$ are relevant, we use the simplified notation $G_{\mathscr{P}}=[x_1,\cdots,x_{n-1}\mid x_n]$.
\end{definition}

\begin{figure}[h]   
    \includegraphics[width=9cm]{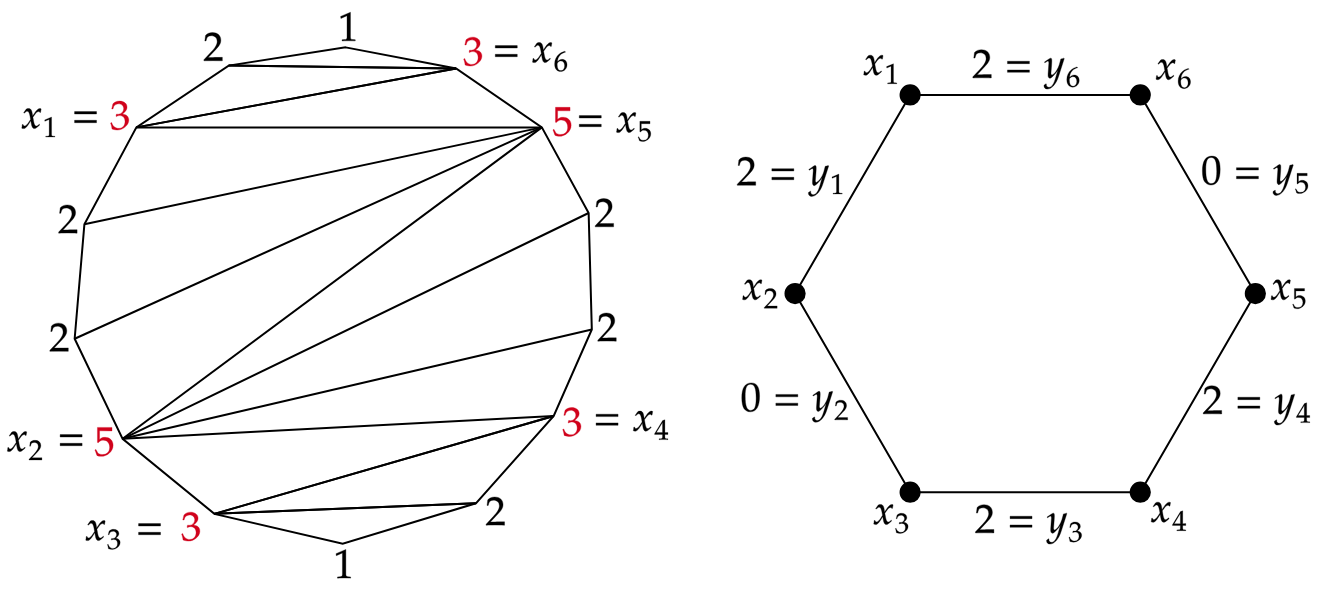}
    \caption{An example of a framed accordion triangulation $\mathscr{P}$ and its coherent graph $G_{\mathscr{P}}$.}
    \label{An example of an accordion triangulation and its coherent graph.}
\end{figure}

\remark Note that the frame accordion triangulations $\mathscr{P}=[1,2,3,2,2,5,3,1,2,3,2,2,5\mid 3]$, $\mathscr{P}_1=[2,3,2,2,5,3,1,2,3,2,2,5,3\mid 1]$, and $\mathscr{P}_2=[3,2,2,5,3,1,2,3,2,2,5,3,1\mid 2]$ have the same coherent graph, see Figure \ref{An example of an accordion triangulation and its coherent graph.}.\\

The coherent graph is defined for any framed triangulated $(s+1)$-gon with $s\geq 4$. However, the most powerful applications of these graphs are for framed accordion triangulations, due to the following results.

\begin{definition} \label{diagonals from x_i go to the pair (x_j,x_{j+1}).}
    Let $\mathscr{P}=[b_1,\cdots,b_s \mid b_0]$ be a framed accordion triangulation and let $G_{\mathscr{P}}=[x_1,\cdots,x_{n-1}\mid x_n]$ be its coherent graph. We say that diagonals from the weight $x_i$ go to the pair $(x_{j \text{ (mod n)}},x_{j+1 \text{( mod n)}})$ if all the diagonals in $\mathscr{P}$ containing the vertex $x_{i}$ are the diagonals connecting $x_i$ with $b_{k}$ for $k_1\leq k \leq k_2$ (strict inequality on the left when $j=i$, and strict inequality on the right when $j=i-1$), where $k_1$ and $k_2$ are the indices such that $b_{k_1}=x_{j}$ and $b_{k_2}=x_{j+1}$. We say that the weight $x_i$ is of type (I) if the diagonals from the weight $x_i$ go to the pair $(x_{j \text{ (mod n)}},x_{j+1 \text{( mod n)}})$ for $j=i$ or $j=i-1$; otherwise the weight $x_i$ is called of type (II).
\end{definition}

\begin{lemma}\label{diagonals fit into weight vertices}
    Let $\mathscr{P}=[b_1,\cdots,b_s\mid b_0]$ be a framed accordion triangulation, and let $G_\mathscr{P}=[x_1,\cdots,x_{n-1}\mid x_n]$ be its coherent graph. If $n\geq 2$, then for each $1\leq i\leq  n$ the weight $x_i$ is either a weight of type (I) or a weight of type (II).
\end{lemma} 
\begin{proof}
    By Definition \ref{Definition of accordion triangulation}, the diagonals from the weight $x_i$ connect the weight $x_i$ to $\text{index}(x_i)-1$ consecutive vertices of $\mathscr{P}$. For each triple $b_{j},b_{j+1},b_{j+2}$ of consecutive vertices  connected to $x_{i}$ in $\mathscr{P}$, we have that $b_{j+1}=2$, so $b_{j+1}$ cannot be another weight. By taking all possible triple of consecutive vertices connected to $x_i$, we conclude that $x_i$ can connect at most $2$ other weights. Since $n\geq 2$, the weight $x_i $ must connect at least another weight. If the weight $x_i$ connects exactly one weight, then $x_i$ is a weight of type (I), otherwise it is a weight of type (II).
\end{proof}

\begin{figure}[H] %
\begin{tabular}{ll}
\hspace*{-1,2em}
\includegraphics[width=8.1cm]{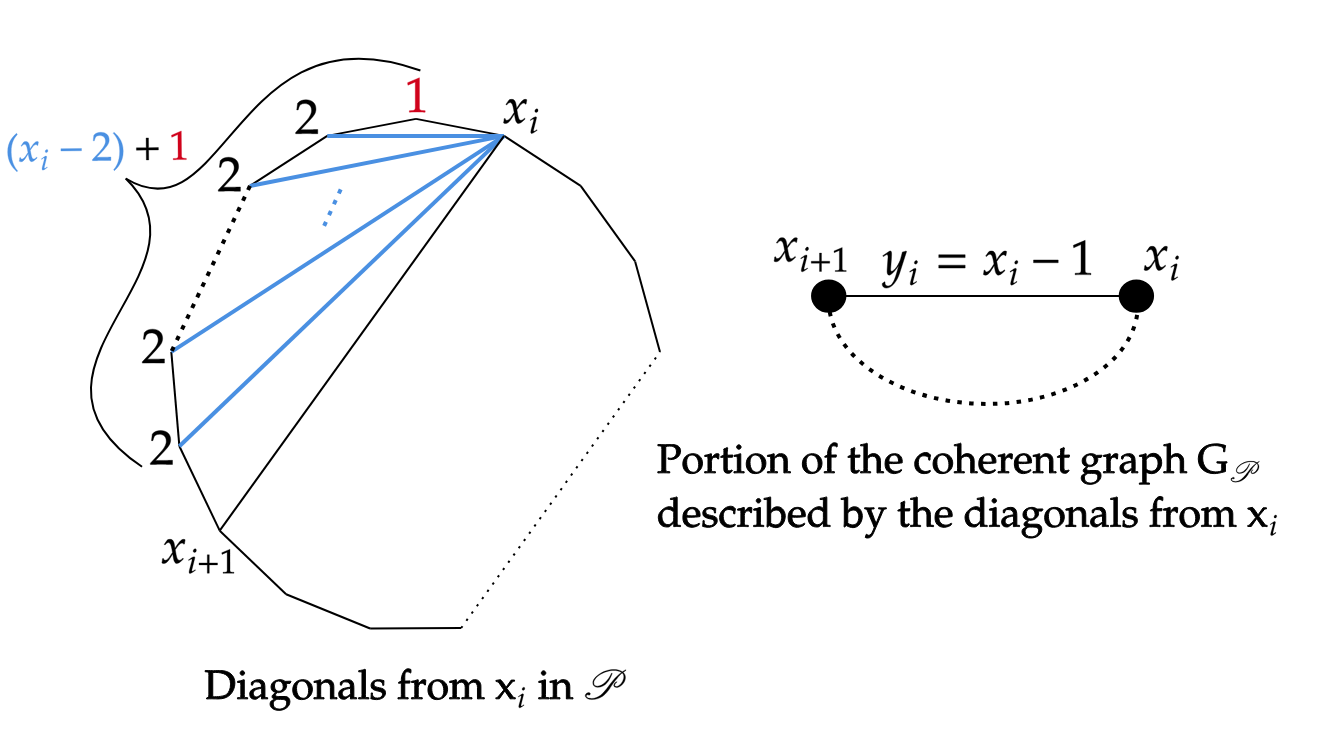}
&\hspace{-2em}
\includegraphics[width=8.1cm]{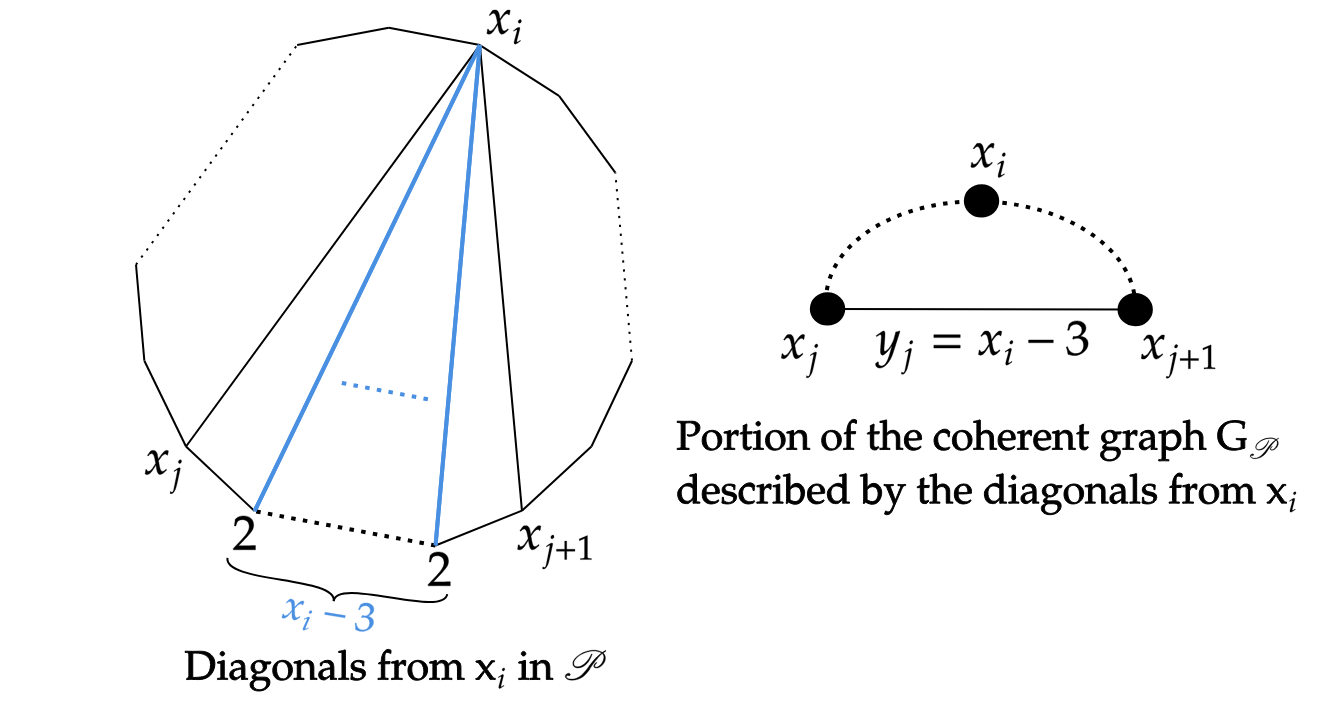}
\end{tabular}
\caption{Geometric interpretation of Lemma \ref{diagonals fit into weight vertices}.}
\label{Geometric interpretation of how we get the linear relations}
\end{figure}

\remark Weights of type (I) correspond (when $j=i$) to the situation on the left in Figure \ref{Geometric interpretation of how we get the linear relations}, while weights of type (II) correspond to the situation on the right in Figure \ref{Geometric interpretation of how we get the linear relations}.

\begin{lemma}(Alternative description of accordion triangulations) \label{Alternative description of accordion triangulations}\\
    Let $n\geq 2$ be an integer, let $x_1,\cdots,x_n$ be integers greater than $2$, and let $\mathscr{P}_{x_1,\cdots
    ,x_n}$ be the family of all accordion triangulations with weights $x_1,\cdots,x_n$. For each $1\leq i\leq n$ and $1\leq j\leq n$, there exists a unique accordion triangulation in $\mathscr{P}_{x_1,\cdots,x_n}$ whose diagonals from $x_i$ go to the pair $(x_{j\text{ (mod n)}}, x_{j+1\text{ (mod n)}})$.
\end{lemma}
\begin{proof}
    By Lemma \ref{diagonals fit into weight vertices}, the diagonals from any weight $x_i$ go to a certain pair $(x_{j \text{ (mod n)}}
    $, $x_{j+1 \text{ (mod n)}})$ allowing $j=i$ or $j=i-1$. 
    \begin{itemize}
        \item $\mathbf{If \hspace{0,2em} x_i \hspace{0,2em} is \hspace{0,2em} a \hspace{0,2em} weight \hspace{0,2em} of \hspace{0,2em} type \hspace{0,2em} (I)}$: the diagonals connecting $x_i$ divide the polygon into two parts: one part that is already triangulated and one that must be triangulated, as in Figure \ref{Geometric interpretation of how we get the linear relations}. If $j=i$, then we apply Lemma \ref{diagonals fit into weight vertices} and Definition \ref{Definition of accordion triangulation} to the weight $x_{j+1}$. Since we are drawing the diagonals of an accordion triangulation, the diagonals from $x_{i+1}$ must go to the pair $(x_{i-1 \text{ (mod n)}},x_{i \text{ (mod n)}})$. Inductively, we apply Lemma \ref{diagonals fit into weight vertices} and Definition \ref{Definition of accordion triangulation} to the last weight that were connected by the diagonals we have drawn. This process obviously finishes in finite steps, and it determines a unique accordion triangulation with weights $x_1,\cdots,x_n$. The case $j=i-1$ is similar. 

        \item $\mathbf{If \hspace{0,2em} x_i \hspace{0,2em} is \hspace{0,2em} a \hspace{0,2em} weight \hspace{0,2em} of \hspace{0,2em} type \hspace{0,2em} (II)}$: the diagonals connecting $x_i$ divide the polygon into three parts: one part that is already triangulated and two parts that must be triangulated, as in Figure \ref{Geometric interpretation of how we get the linear relations}.
        By Lemma \ref{diagonals fit into weight vertices} and Definition \ref{Definition of accordion triangulation} on the weights $x_{j}$ and $x_{j+1}$, we deduce that the diagonals from $x_{j}$ must go to the pair $(x_{i \text{ (mod n)}}, x_{i+1 \text{ (mod n)}})$ and the diagonals from $x_{j+1}$ must go to the pair $(x_{i-1 \text{ (mod n)}}, x_{i \text{ (mod n)}})$. Inductively, we apply Lemma \ref{diagonals fit into weight vertices} and Definition \ref{Definition of accordion triangulation} to the last weight that were connected by the diagonals we have drawn (if one these two sectors that must be triangulated is already triangulated, we end the process on that side), it determines a unique accordion triangulation with weights $x_1,\cdots,x_n$.
    \end{itemize}    
\end{proof}
\remark \label{pair of basic wormhole triangulations differ in 4 positions.} The previous Lemma and the proof of Proposition \ref{basic wormhole are accordion triangulations} imply that if $\mathscr{P}=[b_1,\cdots,b_s\mid b_0]$ is a basic wormhole triangulation and $\widetilde{\mathscr{P}}$ is a companion of $\mathscr{P}$, then $\widetilde{\mathscr{P}}=[b_1',\cdots,b_s'\mid b_0']$ satisfies that $b'_{0}>1$, and $b_{i}=b_{i}'$ for all $0\leq i\leq n$, except at four distinct indices $1\leq\alpha_1,\alpha_2,\beta_1,\beta_2\leq s$ where $b_{\alpha_1}=b_{\beta_1}=b'_{\alpha_2}=b'_{\beta_2}=1$ and $b'_{\alpha_1}=b'_{\beta_1}=b_{\alpha_2}=b_{\beta_2}=2$.\\

The following result is one of the key observations of this paper.
It tells us that if $\mathscr{P}$ is a framed accordion triangulation, then the coherent graph $G_{\mathscr{P}}$ is naturally equipped with an explicit system of linear relations describing the arrangement of the diagonals of the triangulation $\mathscr{P}$.

\begin{lemma}\label{linear relations in the coherent graph} Let $\mathscr{P}$ be a framed accordion triangulation and let $G_{\mathscr{P}}=([x_1,\cdots,x_{n-1}\mid x_n],(y_1,\cdots,y_n))$ be its coherent graph. 
If $n\geq 2$ and the diagonals from the weight $x_i$ go to the pair $(x_{j \text{ (mod n)}}$, \hspace{0,1em}$x_{j+1\text{ (mod n)}})$ for some pair of indices $1\leq i,j\leq n$, then there is a linear system of $n$ relations associated with the graph $G_\mathscr{P}$. Specifically,
    \begin{align}
        y_\ell=x_{\ell_{i,j}}-k_{\ell_{i,j}} \hspace{1em} \text{for} \hspace{1em}1\leq \ell \leq n,
    \end{align}
where each weight in $\mathscr{P}$ is exactly one of the numbers $x_{\ell_{i,j}}$ for some appropriate $\ell$ , and $k_{\ell_{i,j}}$ is equal to $3$ for all $\ell$, except for two special values of $\ell$ for which $k_{\ell_{i,j}}$ is equal to $1$.
\end{lemma}

\begin{proof}
    Since the diagonals from the weight $x_i$ go to the pair $(x_{j \text{ (mod n)}}$, \hspace{0,1em}$x_{j+1\text{ (mod n)}})$, by the proof of Lemma \ref{Alternative description of accordion triangulations}, the arrangement of the diagonals of $\mathscr{P}$ is determined by this data. The set of linear relations come from Lemma \ref{diagonals fit into weight vertices}, see Figure \ref{Geometric interpretation of how we get the linear relations}. 
\end{proof}

\remark By Lemma \ref{diagonals fit into weight vertices}, the diagonals from any weight $x_i$ go to a certain pair $(x_{j \text{ (mod n)}}
$, $x_{j+1 \text{ (mod n)}})$ allowing $j=i$ or $j=i-1$. By Lemma \ref{Alternative description of accordion triangulations}, the system of linear relations in Lemma \ref{linear relations in the coherent graph} does not depend on the choice of the weight $x_i$, it depends only on $\mathscr{P}$. If $n=1$, then the coherent graph of a framed accordion triangulation is also naturally equipped with an explicit system of relations. It satisfies the relation $y_n=x_1+1$.

\begin{example} Revisiting the example in Figure \ref{An example of an accordion triangulation and its coherent graph.}. The system of linear relations associated with $G_{\mathscr{P}}$ (Lemma \ref{linear relations in the coherent graph}) is:
\begin{align*}
    y_1=x_5-3, \hspace{0,5em} y_2=x_4-3, \hspace{0,5em}y_3=x_3-1, \hspace{0,5em} y_4=x_2-3, \hspace{0,5em} y_5=x_1-3, \hspace{0,5em} y_6=x_6-1.
\end{align*}
\vspace*{-2em}
\begin{figure}[H] %
    \includegraphics[width=9cm]{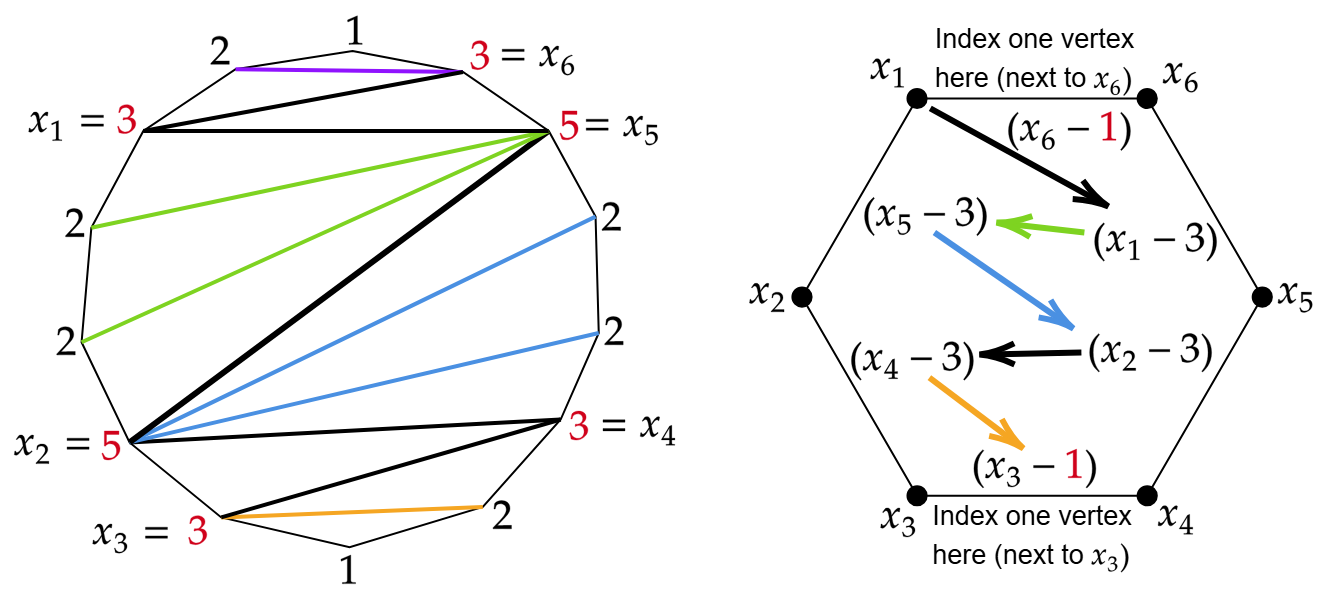}
    \caption{Revisiting the coherent graph in Figure \ref{An example of an accordion triangulation and its coherent graph.}.}
    \label{coherent graph revisited}
\end{figure}
\end{example}

\begin{definition}\label{Definition of equal coherent graphs}
    Let $\mathscr{P}$ and $\widetilde{\mathscr{P}}$ be framed triangulated polygons. We say that the coherent graphs $G_{\mathscr{P}}=([x_1,\cdots,x_{n-1}\mid x_n],(y_1,\cdots,y_n))$ and $G_{\widetilde{\mathscr{P}}}=([\widetilde{x_1},\cdots,\widetilde{x}_{m-1}\mid \widetilde{x}_m],(\widetilde{y}_1,\cdots,\widetilde{y}_m))$ are \textbf{equal} if and only if $n=m$, $x_i=\widetilde{x_i}$ and $y_i=\widetilde{y}_i$ for $1\leq i\leq n$.
\end{definition}

\begin{lemma}\label{coherent graph of basic wormholes are the same}
    Let $\mathscr{P}$ and $\widetilde{\mathscr{P}}$ be basic wormhole triangulations that are not equal. Assume that the hidden index in both framed triangulated polygons is a weight. The coherent graphs $G_{\mathscr{P}}$ and $G_{\widetilde{\mathscr{P}}}$ are equal if and only if $\widetilde{\mathscr{P}}$ is a companion of $\mathscr{P}$.
\end{lemma}
\begin{proof} 
    By Definition \ref{definition of basic wormhole triangulation} and Remark \ref{pair of basic wormhole triangulations differ in 4 positions.}, the converse implication is straightforward. Assume that $G_{\mathscr{P}}$ and $G_{\widetilde{\mathscr{P}}}$ are equal.
    By Definition \ref{definition of basic wormhole triangulation}, $\mathscr{P}$ and $\widetilde{\mathscr{P}}$ are given by one of the extended zero chains associated with them, say $\mathscr{P}=[b_1,\cdots,b_{s_1}\mid b_0]$ and $\widetilde{\mathscr{P}}=[b_1',\cdots,b'_{s_2}\mid b'_0]$ with $b_0,b_0'\geq 3$.
    Denote the coherent graphs by $G_{\mathscr{P}}=([x_1,\cdots,x_{n-1}\mid x_n],(y_1,\cdots,y_n))$ and $G_{\widetilde{\mathscr{P}}}=([\widetilde{x_1},\cdots,\widetilde{x}_{m-1}\mid \widetilde{x}_m],(\widetilde{y}_1,\cdots,\widetilde{y}_m))$.
    By Definition \ref{Definition of equal coherent graphs}, we have that $n=m$, $x_i=\widetilde{x_i}$ and $y_i=\widetilde{y_i}$ for $1\leq i\leq n$. Since $b_0,b_0'\geq 3$, by Definition \ref{coherent graph},  $\mathscr{P}$ and $\widetilde{\mathscr{P}}$ have the same weights in the same positions and $s_1=s_2$. 
    By Lemma \ref{every triangulation with exactly two 1's is an accordion triangulation}, there are exactly two indices $i$ where $b_i=1$ and exactly two indices $j$ where $b_{j}'=1$. Assume that $b_i=1$ for $i\in \{\alpha_1,\beta_1\}$ and $b_{j}'=1$ for $j\in \{\alpha_2,\beta_2\}$.
    Since we know all the weights, by Definition \ref{Definition of accordion triangulation}, the position of $b_{\alpha_1}=1$ determines the position of $b_{\beta_1}=1$, while the position of $b'_{\alpha_2}=1$ determines the position of $b'_{\beta_2}=1$. Furthermore, the numbers $\alpha_1,\alpha_2,\beta_1,\beta_2$ are $4$ different numbers because $\mathscr{P}$ and $\widetilde{\mathscr{P}}$ are not equal (Remark \ref{pair of basic wormhole triangulations differ in 4 positions.}). 
    Then $b_0>1$ and $b_i=b'_i$ for all $0\leq i,i'\leq s$, except in $4$ positions where we are swapping the pairs of $1$'s in each extended zero fraction with a pair of $2's$ in the other one. 
    Therefore, $\mathscr{P}$ and $\widetilde{\mathscr{P}}$ correspond to WW-decompositions of the same fraction. By Definition \ref{definition of basic wormhole triangulation} and Remark \ref{pair of basic wormhole triangulations differ in 4 positions.}, we conclude.
\end{proof}

\begin{example} Consider the framed triangulated polygons $\mathscr{P}=[1,2,3,2,2,5,3,1,2,3,2,2,5\mid 3]$ and $\widetilde{\mathscr{P}}=[2,2,3,1,2,5,3,2,2,3,1,2,5\mid 3]$, see Figure \ref{rotation of diagonals, a companion triangulation}. Note that $\mathscr{P}$ and $\widetilde{\mathscr{P}}$ correspond to WW-decompositions of the WW-sequence $\{2,2,3,2,2,5,3,2,2,3,2,2\}$ with WW-index $x_6=3$.
    \begin{figure}[H] %
        \begin{tabular}{ll}
        \includegraphics[width=9cm]{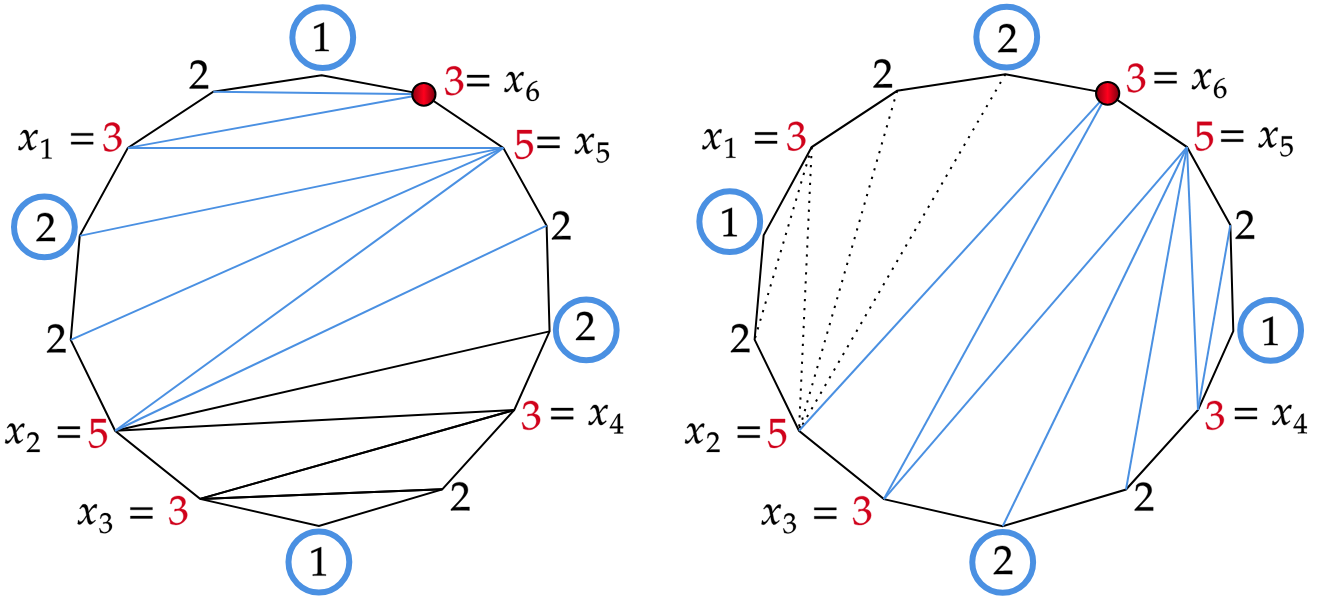}
        \end{tabular}
        \caption{Framed triangulated polygon and a companion.}
        \label{rotation of diagonals, a companion triangulation}
    \end{figure}
The triangulations $\widetilde{\mathscr{P}}$ and $\widetilde{\mathscr{P}}$ have the same indices in the same positions except at $4$ positions where we are swapping a pair of index $1$ vertices with a pair of index $2$ vertices (see Remark \ref{pair of basic wormhole triangulations differ in 4 positions.}).
We can think of the triangulation $\widetilde{\mathscr{P}}$ as a triangulation obtained from $\mathscr{P}$ through a process of rotating the diagonals around the pivot vertex $x_6$. During this process, certain diagonals will no longer contribute to the triangulation by connecting vertices. However, these diagonals will reappear in a "coherent" manner in the remaining part of the polygon that stills needs to be triangulated (see Lemma \ref{Alternative description of accordion triangulations}). This idea motivates the term coherent graph, and it will be formalized by the definition of coherent rotation of diagonals (Definition \ref{Standard family of accordion triangulation of weight n and coherent rotation of diagonals for elements of the standard family of weight n}).
\end{example}

Recall that accordion triangulations do not come with a frame. However, we can take a frame that is almost canonical.
\begin{definition} \label{standard framed}
    Let $\mathscr{P}$ be an accordion triangulation. A standard frame for $\mathscr{P}$ is a frame of $\mathscr{P}$ where the hidden vertex is a weight of type (I) and the first entry of the extended zero chain is $1$.\\
\end{definition}
\vspace*{-2em}
\begin{definition}
    \label{Standard family of accordion triangulation of weight n and coherent rotation of diagonals for elements of the standard family of weight n}
    Let $n\geq 2$ be an integer, let $x_{1}, \cdots,x_{n}$ be $n$ integer variables with the restriction that $x_{i}\geq 3$. Let $\mathscr{P}^{m}_{x_1,\cdots,x_n}$ be the unique accordion triangulation with weights $x_1,\cdots,x_n$ (in that order) such that the diagonals from $x_{n}$ go to the pair $(x_{m \text{ (mod n)}},x_{m+1 \text{ (mod n)}})$ for $0\leq m\leq n-1$. 
    The frame of $\mathscr{P}^{m}_{x_1,\cdots,x_n}$ is the one given by the ordering $x_1,\cdots,x_n$ and hidden index $x_n$. 
    We denote by $\mathscr{P}_m$ the family of all triangulations $\mathscr{P}^{m}_{x_1,\cdots,x_n}$. 
    The family $\mathscr{P}_0$ is called the standard family of accordion triangulations, whereas the family $\mathscr{P}_m$ for $1\leq m \leq n-1$ is called the coherent family obtained after $m$-coherent rotations of diagonals from the standard family $\mathscr{P}_0$. We say that $\mathscr{P}^{m}_{x_1,\cdots,x_n}$ in the family $\mathscr{P}_m$ is obtained from $\mathscr{P}^{0}_{x_1,\cdots,x_n}$ via $m$-coherent rotation of diagonals. 
    For any element $\mathscr{P}^{m}_{x_1,\cdots,x_n}$ in the family $\mathscr{P}_m$, we denote by $S_m$ the system of linear relations associated with the coherent graph $G_{\mathscr{P}^{m}_{x_1,\cdots,x_n}}$ where $0\leq m\leq n-1$. 
\end{definition}

Given $n\geq 2$ an integer and an integer $0\leq m\leq n-1$, it is straightforward to explicitly calculate the linear system of $n$ relations associated with $G_{\mathscr{P}^{m}_{x_1,\cdots,x_n}}$ in Lemma $\ref{linear relations in the coherent graph}$.
Specifically
\begin{align}\label{System of equations Sm}
    S_m\colon y_i=x_{(n-i)+m \text{(mod n)}}-k_{i}^{(m)} \hspace{1em} \text{for} \hspace{1em}1\leq i\leq n,
\end{align}
where the numbers $k_{i}^{(m)}$ are given as follows:
\begin{itemize}
    \item $\mathbf{If \hspace{0,2em} n \hspace{0,2em} is \hspace{0,2em} odd}$: $k^{(m)}_{\frac{n-1}{2}+\lceil{\frac{m}{2}}\rceil \text{ (mod n)}}=1$, $k^{(m)}_{n+\lfloor{\frac{m}{2}}\rfloor \text{ (mod n)}}=1$, otherwise $k_{i}^{(m)}=3$,
    \item $\mathbf{If \hspace{0,2em}n \hspace{0,2em} is \hspace{0,2em} even}$:  $k_{\frac{n}{2}+\lfloor \frac{m}{2}\rfloor \text{ (mod n)}}^{(m)}=1, k_{n+\lfloor \frac{m}{2}\rfloor \text{ (mod n)}}^{(m)}=1$, otherwise $k_i^{(m)}=3$,
\end{itemize}
where $\lceil \cdot\rceil$ and $\lfloor \cdot \rfloor$ represent the ceiling and floor functions respectively.

\begin{lemma}\label{criterion for solution for the system S_0 and S_m}
    Let $P\colon \mathbb{Z}^{n}\to \mathbb{Z}^n$ be the map given by $P(x_1,\cdots,x_n)= (x_{n},x_1,\cdots,x_{n-1})$. Then the system of equations
\begin{align}\label{equation (I-P^m)x=v}
    (I-P^m)\vec{x}=\vec{v}.
\end{align} 
has a solution if and only if $\sum_{i\in C_j} v_i=0$ for each cycle $C_j$ in the cycle decomposition induced by $P^m$.
\end{lemma}
\begin{proof}
    The map $P^m$ acts permuting the entries of vectors $(x_1,\cdots,x_n)$, so it breaks the set $\{1,\cdots,n\}$ into disjoint cycles, each of length $\ell=\frac{n}{d}$ where $d=\gcd(m,n)$. If there exists a solution of the equation, then for each component $x_i$ of $\vec{x}$, we have $P^m x_i=x_{i+m}$ where the indices are considered $\text{mod n}$. The equation (\ref{equation (I-P^m)x=v}) implies that
    \begin{align*}
        x_i-x_{i+m \text{ (mod n)}} & =v_i,\\
        x_{i+m \text{ (mod n)}}-x_{i+2m \text{ (mod n)}}& =v_{i+m \text{ (mod n)}},\\
        \vdots\\
        x_{i+(\ell-1)m \text{ (mod n)}}-x_{i}  & =v_{i+(l-1)m \text{ (mod n)}}.
    \end{align*}
    Then $$0=\sum_{j=0}^{\ell-1} v_{i+jm \text{ (mod n)}}.$$
    The converse implication is straightforward.
\end{proof}

\remark \label{one parameter for each cycle in the decomposition} When the system (\ref{equation (I-P^m)x=v}) is consistent, the proof of the previous Lemma gives an explicit parametric solution of this system (1-parameter for each cycle).  For example, for indices in the orbit containing the number $i$ we can write all the terms in the corresponding orbit of $x_i$ as follows:
\begin{align*}
    x_{i+m \text{ (mod n)}} & =x_{i}-v_i\\
    x_{i+2m \text{ (mod n)}} & =x_i-v_i-v_{i+m \text{ (mod n)}}\\
    & \vdots\\
    x_{i+(\ell-1)m} & = x_{i}-v_i-v_{i+m \text{ (mod n)}}-\ldots-v_{i+(\ell-2)m \text{ (mod n)}.}
\end{align*}
Since the length of each cycle is $\ell =\frac{n}{\gcd(m,n)}$, we deduce that the number of parameters in the parametric solution of the equation (\ref{equation (I-P^m)x=v}) is exactly $\gcd(m,n)$.

In particular, for the systems of equations $S_0$ and $S_{m}$ in (\ref{System of equations Sm}),
\begin{align*}
    S_0:y_i & =x_{n-i \text{ (mod n)}}-k^{(0)}_{i} \hspace{1em} \text{for} \hspace{1em} 1\leq i\leq n, \\
    S_m:y_i & =x_{(n-i)+m \text{ (mod n)}}-k_i^{(m)} \hspace{1em} \text{for} \hspace{1em} 1\leq i \leq n,
\end{align*}
we obtain the system of relations:
$$x_{n-i \text{ (mod n)}}-x_{(n-i)+m \text{ (mod n)}}=k^{(0)}_i-k_{i}^{(m)} \hspace{1em} \text{for} \hspace{1em} 1\leq i\leq n.$$
Changing indices, we get:
$$x_{i}-x_{i+m \text{ (mod n)}}=\underbrace{k^{(0)}_{n-i\text{ (mod n)}}-k^{(m)}_{n-i \text{ (mod n)}}}_{v_i} \hspace{1em} \text{for} \hspace{1em} 1\leq i \leq n.$$
Using the same argument as before, we obtain a parametric description of the orbit of $x_i$.

\begin{theorem}\label{Main theorem}
Let $n\geq 2$ be a integer, and let $\mathscr{P}$ be a framed accordion triangulation with a standard frame and weights $x_1,\cdots,x_n$. Then $\mathscr{P}$ is a basic wormhole triangulation with a standard frame if there exists an integer $1\leq m\leq n-1$ such that the system $S_0\cup S_m$ of $2n$ linear equations
\begin{align}
    S_0:y_i & =x_{n-i \text{(mod n)}}-k_{i}^{(0)}  \hspace{4em} \text{for } 1\leq i\leq n,\\
    S_m:y_i & =x_{(n-i)+m \text{ (mod n)}}-k_i^{(m)} \hspace{1.8em} \text{for } 1\leq i\leq n,
\end{align}
is consistent; where $S_0$ is the system of linear relations associated with the coherent graph $G_{\mathscr{P}^{0}_{x_1,\cdots,x_n}}$, and $S_{m}$ is the system of linear relations associated with the coherent graph $G_{\mathscr{P}^{m}_{x_{1}\cdots,x_{n}}}$. 
The system of equations $S_0\cup S_m$ is inconsistent if and only if $\gcd(n,m) \nmid n-\lfloor\frac{m}{2}\rfloor$ and $\gcd(n,m)\nmid \frac{n}{2}-\lfloor\frac{m}{2}\rfloor$ when $n$ is even; $\gcd(n,m)\nmid n-\lceil\frac{m}{2}\rceil$ and $\gcd(n,m)\nmid \frac{n-1}{2}- \lfloor\frac{m}{2}\rfloor$ when $n$ is odd. When the system of equations $S_0\cup S_m$ is consistent it has $\gcd(n,m)$-parameters, $x_1,\cdots,x_{\gcd(n,m)}$, and its parametric solutions is given by
$$x_{i+jm \text{ (mod n)}}=x_i-(k^{(0)}_{n-i \text{ (mod n)}}-k^{(m)}_{n-i \text{ (mod n)}})- \ldots- (k^{(0)}_{n-(i+(j-1)m) \text{ (mod n)}}-k^{(m)}_{n-(i+(j-1)m) \text{ (mod n)}})$$
for $1\leq i\leq \gcd(n,m)$ and $1\leq j\leq \frac{n}{\gcd(n,m)}$. The numbers $k^{(0)}_{i}$ and $k_{i}^{(m)}$ are given as follows: 
\begin{itemize}
    \item $\mathbf{If \hspace{0,2em} n \hspace{0,2em} is \hspace{0,2em} odd}$: $k^{(0)}_\frac{n-1}{2}=1$, $k^{(0)}_{n}=1$, $k^{(m)}_{\frac{n-1}{2}+\lceil{\frac{m}{2}}\rceil \text{ (mod n)}}=1$, $k^{(m)}_{n+\lfloor{\frac{m}{2}}\rfloor \text{ (mod n)}}=1$, otherwise $k_{i}^{(0)}=3$ and $k_{i}^{(m)}=3$.
    
    \item $\mathbf{If \hspace{0,2em}n \hspace{0,2em} is \hspace{0,2em} even}$: $k^{(0)}_{\frac{n}{2}}=1$, $k^{(0)}_{n}=1$, $k_{\frac{n}{2}+\lfloor \frac{m}{2}\rfloor \text{ (mod n)}}^{(m)}=1, k_{n+\lfloor \frac{m}{2}\rfloor \text{ (mod n)}}^{(m)}=1$, otherwise $k_i^{(0)}=3$ and $k_i^{(m)}=3$.
\end{itemize}
The parametric solution of $S_0\cup S_m$ gives a pair ($\mathscr{P}^{0}_{x_1,\cdots,x_n} \hspace{0,3em} , \hspace{0,3em} \mathscr{P}^{m}_{x_1,\cdots,x_n}$) of parametric basic wormhole triangulations with $n$ weights where $\mathscr{P}=\mathscr{P}^{0}_{x_1,\cdots,x_n}$ and $\mathscr{P}^{m}_{x_1,\cdots,x_n}$ is a companion of $\mathscr{P}$.
\end{theorem}

\begin{proof}
    By Definition \ref{standard framed}, we have that $\mathscr{P}= \mathscr{P}^{0}_{x_1,\cdots,x_n}$, so $\mathscr{P}$ is an element of the standard family $\mathscr{P}_0$. Let $\mathscr{P}^{\vee}$ be any companion of $\mathscr{P}$. Since $\mathscr{P}$ is a basic wormhole triangulation with a standard frame, by Lemma \ref{coherent graph of basic wormholes are the same}, the coherent graphs $G_\mathscr{P}$ and $G_{\mathscr{P}^{\vee}}$ are equal.
    Therefore, we have that $G_{\mathscr{P}}=([x_1,\cdots,x_{n-1}\mid x_n],(y_1,\cdots,y_n))$ and $G_{\mathscr{P}^{\vee}}=([x_1,\cdots,x_{n-1}\mid x_n],(y_1,\cdots,y_n))$.
    Since $\mathscr{P}^{\vee}$ is also a basic wormhole triangulation, by Lemma \ref{basic wormhole are accordion triangulations}, we have that $\mathscr{P}^{\vee}$ is a framed accordion triangulation. 
    By Lemma \ref{Alternative description of accordion triangulations}, the triangulation $\mathscr{P}^{\vee}$ is completely determined by knowing $0\leq m \leq n$ such that the diagonals from $x_n$ go to the pair $(x_{m\text{ (mod n)}},x_{m+1 \text{(mod n)}})$. Note that the cases $m=0$ and $m=n$ correspond to the triangulation $\mathscr{P}$, so we discard these cases because $\mathscr{P}$ and $\mathscr{P}^{\vee}$ are not equal.
    By Definition \ref{Standard family of accordion triangulation of weight n and coherent rotation of diagonals for elements of the standard family of weight n}, each choice of $m$ implies that $\mathscr{P}^{\vee}=\mathscr{P}^{m}_{x_1,\cdots,x_n}$ for some $1\leq m\leq n-1$. 
    By Lemma \ref{linear relations in the coherent graph}, the coherent graphs $G_{\mathscr{P}}$ and $G_{\mathscr{P}^{\vee}}$ have associated a linear system of relations describing the arrangement of the diagonals in $\mathscr{P}$ and $\mathscr{P}^{\vee}$, respectively.
    By Lemma \ref{coherent graph of basic wormholes are the same}, the system of equations $S_0$ and $S_m$ must be simultaneously consistent. If the system of $2n$ linear relations $S_0\cup S_m$ is not consistent, then $\mathscr{P}^{m}_{x_1,\cdots,x_n}$ cannot be equal to $\mathscr{P}^{\vee}$ because of Lemma \ref{coherent graph of basic wormholes are the same}. For each $1\leq i\leq n$, substituting the equation of $y_i$ in $S_0$ into the equation of $y_i$ in $S_m$ (see remark \ref{System of equations Sm}), we obtain the system of relations
    \begin{align}\label{system of relations, main theorem}
        (I-P^m)\vec{x}=k^{(0)}-k^{(m)},
    \end{align}
    where $\vec{x}=(x_1,\cdots,x_n)$, $(k^{(0)}-k^{(m)})_i = k_{n-i \text{ (mod n)}}^{(0)}-k_{n-i \text{ (mod n)}}^{(m)}$, $I$ is the identity map and $P\colon \mathbb{Z}^n\to \mathbb{Z}^{n}$ is the map given by $P(x_1,\cdots,x_n)=(x_{n},x_1,\cdots,x_{n-1})$. 
    By Lemma \ref{criterion for solution for the system S_0 and S_m}, this system is consistent if and only if $\sum_{i\in C_j} (k^{(0)}-k^{(m)})_i=0$ for each cycle $C_j$ in the cycle decomposition induced by $P^{m}$.
    When $2\leq m\leq n-1$, the vector $k^{(0)}-k^{(m)}$ has all the entries equal to $0$, except at exactly four distinct positions: at the two positions where $k^{(0)}_i=1$, in which case $(k^{(0)}-k^{(m)})_i=-2$, and at the two positions where $k^{(m)}_j=1$, in which case $(k^{(0)}-k^{(m)})_j=2$. There is a degenerate case when $n$ is odd and $m=n-1$, where exactly one position corresponds to $-2$ and one position corresponds to $+2$.
    When $m=1$, if $n$ is odd, there is exactly one position corresponding to $-2$ and one position corresponding to $+2$, and if $n$ is even, the vector $k^{(0)}-k^{(m)}$ is the zero vector.
    Therefore, to check the consistency of the system (\ref{system of relations, main theorem}), we must verify
    that, for each orbit, the number of indices $i$ such that $(k^{(0)}-k^{(m)})_i=2$ is equal to the number of indices $j$ such that $(k^{(0)}-k^{(m)})_j=-2$.    
    \begin{itemize}
    \item $\mathbf{If \hspace{0,2em} n \hspace{0,2em} is \hspace{0,2em} odd}$: $k^{(0)}_\frac{n-1}{2}=1$, $k^{(0)}_{n}=1$, $k^{(m)}_{\frac{n-1}{2}+\lceil{\frac{m}{2}}\rceil \text{ (mod n)}}=1$, and $k^{(m)}_{n+\lfloor{\frac{m}{2}}\rfloor \text{ (mod n)}}=1$.
    \begin{enumerate}
        \item If $x_{\frac{n+1}{2}}$ and $x_{\frac{n+1}{2}-\lceil \frac{m}{2} \rceil \text{(mod n)}}$ have the same orbit, there is an integer $\lambda_1$ such that $\frac{n+1}{2}+m\lambda_1\equiv \frac{n+1}{2}-\lceil \frac{m}{2} \rceil $(mod n), i.e. $m\lambda_1\equiv n-\lceil \frac{m}{2} \rceil$ (mod n). This equation has a solution if and only if $\gcd(n,m)\mid n- \lceil \frac{m}{2} \rceil$.
        \item If $x_{\frac{n+1}{2}}$ and $x_{n-\lfloor \frac{m}{2} \rfloor \text{(mod n)}}$ have the same orbit, there is an integer $\lambda_2$ such that $\frac{n+1}{2}+m\lambda_2\equiv n -\lfloor \frac{m}{2} \rfloor$(mod n). It holds if and only if $\gcd(n,m) \mid \frac{n-1}{2}-\lfloor \frac{m}{2} \rfloor$.
        \item If $x_n$ and $x_{\frac{n+1}{2}-\lceil\frac{m}{2}\rceil \text{(mod n)}}$ have the same orbit, there is an integer $\lambda_3$ such that $n+m\lambda_3\equiv \frac{n+1}{2}-\lceil \frac{m}{2}\rceil$(mod n). It holds if and only if $\gcd(n,m)\mid \frac{n+1}{2}-\lceil \frac{m}{2} \rceil$.
        \item If $x_n$ and $x_{n-\lfloor \frac{m}{2} \rfloor \text{(mod n)}}$ have the same orbit, there is an integer $\lambda_4$ such that $n+m\lambda_4\equiv n-\lfloor \frac{m}{2}\rfloor$ (mod n). It holds if and only if $\gcd(n,m)\mid n-\lfloor \frac{m}{2} \rfloor$.
    \end{enumerate}
    
    \item $\mathbf{If \hspace{0,2em}n \hspace{0,2em} is \hspace{0,2em} even}$: $k^{(0)}_{\frac{n}{2}}=1$, $k^{(0)}_{n}=1$, $k_{\frac{n}{2}+\lfloor \frac{m}{2}\rfloor \text{ (mod n)}}^{(m)}=1,$ and $k_{n+\lfloor \frac{m}{2}\rfloor \text{ (mod n)}}^{(m)}=1$.
    \begin{enumerate}
        \item If $x_{\frac{n}{2}}$ and $x_{\frac{n}{2}-\lceil\frac{m}{2}\rceil\text{(mod n)}}$ have the same orbit, there is an integer $\lambda_1'$ such that $\frac{n}{2}+m\lambda'_1\equiv \frac{n}{2}-\lfloor \frac{m}{2} \rfloor$ (mod n). It holds if and only if $\gcd(n,m)\mid n-\lfloor \frac{m}{2}\rfloor$.
        \item If $x_{\frac{n}{2}}$ and $x_{n-\lfloor \frac{m}{2} \rfloor (\text{mod n})}$ have the same orbit, there is an integer $\lambda_2'$ such that $\frac{n}{2}+m\lambda'_2\equiv n-\lfloor \frac{m}{2}\rfloor$ (mod n). It holds if and only if $\gcd(n,m)\mid \frac{n}{2}-\lfloor \frac{m}{2}\rfloor$.
        \item If $x_n$ and $x_{\frac{n}{2}-\lfloor \frac{m}{2}\rfloor \text{(mod n)}}$ have the same orbit, there is an integer $\lambda'_3$ such that $n+m\lambda'_3\equiv \frac{n}{2}-\lfloor\frac{m}{2}\rfloor$(mod n). It holds if and only if $\gcd(n,m)\mid \frac{n}{2}-\lfloor\frac{m}{2}\rfloor$.
        \item If $x_n$ and $x_{n-\lfloor \frac{m}{2}\rfloor}$ have the same orbit, there is an integer $\lambda'_4$ such that $n+m\lambda'_4\equiv n-\lfloor \frac{m}{2}\rfloor$(mod n). It holds if and only if $\gcd(n,m)\mid n-\lfloor \frac{m}{2} \rfloor$. 
    \end{enumerate}   
\end{itemize}
Since $m$ is an integer, we have that $\lceil\frac{m}{2}\rceil+\lfloor \frac{m}{2}\rfloor= m$. The equation in $(1)$ has a solution if and only if the equation in $(4)$ has a solution, and the equation in $(2)$ has a solution if and only if the equation in $(3)$ has a solution. When the  system $S_0\cup S_m$ is consistent, by Remark \ref{one parameter for each cycle in the decomposition}, we have a parametric solution with $\gcd(n,m)$-parameters $x_1,\cdots,x_{\gcd(n,m)}$. 
    On the other hand, for each solution of $S_0\cup S_m$ we can construct the framed accordion triangulations $\mathscr{P}^{0}_{x_1,\cdots,x_n}$ and $\mathscr{P}^{m}_{x_1,\cdots,x_n}$. Since $\mathscr{P}$ has a standard frame, we have that $\mathscr{P}=\mathscr{P}^{0}_{x_1,\cdots,x_n}$. Therefore $\mathscr{P}^{m}_{x_1,\cdots,x_n}$ is a companion of $\mathscr{P}$, so $\mathscr{P}$ is a basic wormhole triangulation.
\end{proof}

\begin{corollary} \label{Basic wormhole triangulations algorithm} (Basic wormhole triangulations algorithm).\\
    \textbf{Input}: An integer $n\geq 2$.\\
    \textbf{Output}: All parametric families of basic wormhole triangulations with $n$ weights.
    \begin{itemize}
        \item{\textbf{Step I}}: Let $x_1,\cdots,x_n$ be $n$ integer variables with the condition $x_{i}\geq 3$ for every $1\leq i\leq n.$ Let $S_0$ be the linear system of equations associated with the coherent graph $G_{\mathscr{P}^{0}_{x_1,\cdots,x_n}}$, where $\mathscr{P}^{0}_{x_1,\cdots,x_n}$ is an element of the standard family of accordion triangulations.
    Set $m=1$ and continue to Step II. 
    \item{\textbf{Step II}}: 
    \begin{enumerate}
        \item If $m\leq n-1$: we denote by $S_m$ the linear system of equations associated with the coherent graph $G_{\mathscr{P}^{m}_{x_1,\cdots,x_n}}$, where $\mathscr{P}^{m}_{x_1,\cdots,x_n}$ is obtained from $\mathscr{P}^{0}_{x_1,\cdots,x_n}$ via $m$-coherent rotation of diagonals. By Theorem \ref{Main theorem}, we can check when the  system $S_0\cup S_m$ is consistent.
        \begin{enumerate}
            \item If the system $S_0\cup S_m$ is inconsistent: the pair ($\mathscr{P}^{0}_{x_1,\cdots,x_n} \hspace{0,3em} , \hspace{0,3em} \mathscr{P}^{m}_{x_1,\cdots,x_n}$) does not give a pair of basic wormhole triangulations for any choice of $x_1,\cdots,x_n$ with $x_i\geq 3$ and $1\leq i \leq n$. Redefine $m$ as $m+1$ and return to Step II.
            \item If the system $S_0\cup S_m$ is consistent: the explicit parametric description of the solutions gives a pair ($\mathscr{P}^{0}_{x_1,\cdots,x_n} \hspace{0,3em} , \hspace{0,3em} \mathscr{P}^{m}_{x_1,\cdots,x_n}$) of parametric basic wormhole triangulations with $n$ weights  where $\mathscr{P}=\mathscr{P}^{0}_{x_1,\cdots,x_n}$ and  $\mathscr{P}^{m}_{x_1,\cdots,x_n}$ is a companion of $\mathscr{P}$. Define $n_m:=m$, and then redefine $m$ as $m+1$ and return to Step II.
        \end{enumerate}
        \item If $m=n$: Go to step III.
    \end{enumerate}
    \item{\textbf{Step III}}: By applying the same cyclic permutation to the numbers in the extended zero chains of both parametric basic wormhole triangulations  ($\mathscr{P}^{0}_{x_1,\cdots,x_n}$, $\mathscr{P}^{n_m}_{x_1,\cdots,x_n})$ arising from the consistent systems in Step II, we describe all the basic wormhole triangulations with $n$ weights. 
    The only condition is that the new hidden index, after the cyclic permutation of the entries in both extended zero chains, must be greater than $1$ in both extended zero chains.
\end{itemize}
\end{corollary}
\begin{proof}
    Let $\mathscr{P}$ be a basic wormhole triangulation with $n$ weights. By Definition \ref{definition of basic wormhole triangulation}, $\mathscr{P}$ must be of the form $\mathscr{P}=[b_1,\cdots,b_s\mid b_0]$ for some unknown $b_i$. 
    It must come from a WW-sequence $\{b'_1,\cdots,b'_s\}$ for some unknown $b'_{i}>1$. Via cyclic permutation (Definition \ref{m-cyclic permutation}) of $\{b'_1,\cdots,b'_s\}$ we can construct a WW-sequence with WW-index $b_{0}''\geq 3$ and such that one of its WW-decompositions correspond to the underlying triangulated polygon defined by $\mathscr{P}$ but with a standard frame. This is possible since Lemma \ref{bound for the number of WW-decompositions under cyclic permutation} tells us that in the process we do not lose WW-decompositions, and at the level of WW-decompositions this corresponds to a cyclic permutation of the numbers in the zero chains. By Remark \ref{cyclically rotate entries of an extended zero chain corresponds to choose another hidden index}, this corresponds to a change of hidden index of the corresponding triangulated polygons. By Theorem \ref{Main theorem}, we have an explicit parametric description of all basic wormhole triangulations with a standard frame. Since the WW-decompositions of the last WW-sequence are determined, we just reverse the cyclic permutation. It corresponds to a change of frame in the associated triangulated polygons (that have already been determined). Therefore, $\mathscr{P}$ is determined by one of the basic wormhole triangulations with standard frame and an appropriate change of hidden index. The condition in Step III specifies the possible changes in the hidden index for the pair ($\mathscr{P}^{0}_{x_1,\cdots,x_n}$, $\mathscr{P}^{n_m}_{x_1,\cdots,x_n})$.
\end{proof}
\remark The condition in Step III ensures that cyclic permutations of the WW-sequences defined by the pair  ($\mathscr{P}^{0}_{x_1,\cdots,x_n}$, $\mathscr{P}^{n_m}_{x_1,\cdots,x_n})$ is a WW-sequence.\\

As an application of the new techniques introduced in this paper we present a simple alternative proof of \cite{Flipping}*{Theorem 4.3}.
\begin{theorem}\label{Alternative proof of HTU theorem on maximal number of extremal P-resolutions}
    Let $0<q<\Delta$ be coprime integers. The cyclic quotient singularity $\frac{1}{\Delta}(1,q)$ can admit at most two distinct extremal P-resolutions.
\end{theorem}

\begin{proof}
    Assume that the cyclic quotient singularity $\frac{1}{\Delta}(1,q)$ admits $J\geq 3$ distinct extremal P-resolutions. That is, the Hirzebruch-Jung continued fraction of $\frac{\Delta}{\Delta-q}$ defines a WW-sequence with $J\geq 3$ distinct WW-decompositions. Let $b_0$ be the WW-index (Definition \ref{WW-index}) of this WW-sequence.
    Due to the HTU algorithm, Lemma \ref{HTU algorithm} and Remark \ref{Degenerate case has 2 WW-decompositions}, it is sufficient to assume that $b_0>1$. Via $r$-cyclic permutation (Definition \ref{m-cyclic permutation}) of the WW-sequence (allowing $r=0$), we can construct a WW-sequence with a WW-index $b_0'\geq 3$, where one of the associated triangulations is a framed accordion triangulation $\mathscr{P}$ with a standard frame. By Lemma \ref{bound for the number of WW-decompositions under cyclic permutation}, $J\leq J'$ where $J'$ is the number of WW-decompositions of the last WW-sequence. To complete the proof, it is sufficient to show that $J'\geq 3$ is impossible.
    Pick any three WW-decompositions of the last WW-sequence (one of them being the one corresponding to the triangulation $\mathscr{P}$).
    By Theorem \ref{Main theorem}, the system of equations $S_0\cup S_m\cup S_{m'}$ must be simultaneously consistent for $1\leq m\neq m' \leq n$, where $n$ is the number of weights of the triangulations. 
    As in the proof of Theorem \ref{Main theorem}, we can construct linear system of relations of the form
    \begin{align}\label{system 1 in alternative proof}
        (I-P^{m})\vec{x} & = k^{(0)}-k^{(m)},
    \end{align}
    \begin{align}\label{system 2 in alternative proof}
        (I-P^{m'})\vec{x} & = k^{(0)}-k^{(m')},
    \end{align}
    where $\vec{x}=(x_1,\cdots,x_n)$, $(k^{(0)}-k^{(m)})_i = k_{n-i \text{ (mod n)}}^{(0)}-k_{n-i \text{ (mod n)}}^{(m)}$, $(k^{(0)}-k^{(m')})_i = k_{n-i \text{ (mod n)}}^{(0)}-k_{n-i \text{ (mod n)}}^{(m')}$, $I$ is the identity map and $P\colon \mathbb{Z}^n\to \mathbb{Z}^{n}$ is the map given by $P(x_1,\cdots,x_n)=(x_{n},x_1,\cdots,x_{n-1})$. Subtracting the equations of (\ref{system 2 in alternative proof}) from the equations of (\ref{system 1 in alternative proof}), we obtain the system of relations
    \begin{align}\label{system 3 in alternative proof}
        (P^{m'}-P^{m})\vec{x}=k^{(m')}-k^{(m)}.
    \end{align}
    We can construct a constraint graph via the system of relations (\ref{system 1 in alternative proof}),(\ref{system 2 in alternative proof}), and (\ref{system 3 in alternative proof}) as follows: Given the system of inequalities $x_i-x_j\leq k_{ij}$, we define the weighted, directed graph $G$, where the set vertices consists of one vertex $x_i$ for each coordinate of $\vec{x}$ and one additional vertex $x_0$. For each inequality $x_i-x_j\leq k_{ij}$ we have a directed edge from $x_j$ to $x_i$ with weight equal to $k_{ij}$, and for each $x_i$ we have a directed edge from $x_0$ to $x_i$ with weight equal to $0$. 
    By \cite{ConstraintsGraph}*{Thm 24.9}, if $G$ contains a negative-weight cycle, then there is no feasible solution for these systems. Therefore, to obtain the contradiction, it is sufficient to show that $G$ contains a negative-weight cycle. As in Theorem \ref{Main theorem}, the vectors $k^{(0)}-k^{(m)}$ and $k^{(0)}-k^{(m')}$ have exactly $4$ entries that are not equal to $0$: two entries equal to $2$ and two entries equal to $-2$ (there are degenerate cases when $m$ and $m'$ take the values $1$ or $n-1$. However, the argument below remains valid, since we assume $m\neq m'$. In these cases, it may be necessary to swap the roles of the systems (\ref{system 1 in alternative proof}) and (\ref{system 2 in alternative proof})). 
    Let $i$ be an index such that $(k^{(0)}-k^{(m)})_i=-2$. We will show that there is a negative-weight cycle based at $x_{n-i \text{ (mod n)}}$. By Lemma \ref{criterion for solution for the system S_0 and S_m}, the orbit of $x_{n-i\text{ (mod n)}}$ must contain a $x_j$ such that $(k^{(0)}-k^{(m)})_{n-j}=2$ and the system (\ref{system 1 in alternative proof}) defines a $0$-weighted cycle based at $x_{n-i\text{ (mod n)}}$. To construct the negative loop, we make a slight modification to this 0-weighted cycle based at $x_{n-i\text{ (mod n)}}$.
    Since $m\neq m'$, we can exclude the $+2$-edge in this $0$-weighted cycle by using a $0$-edge connecting the orbit of $x_{n-i\text{ (mod n)}}$ with respect to (\ref{system 1 in alternative proof}) with the orbit of $x_{n-i \text{ (mod n)}}$ with respect to (\ref{system 2 in alternative proof}), i.e., by using the system (\ref{system 3 in alternative proof}). Then we use some $0$-edges in the orbit of $x_{n-i\text{ (mod n)}}$ with respect to (\ref{system 2 in alternative proof}) until we can exclude $x_j$ from the cycle (i.e., ensure that our path does not pass through $x_j$). Once this is achieved, we return to the orbit of $x_{n-i \text{ (mod n)}}$ with respect to  (\ref{system 1 in alternative proof}) to complete the cycle.
\end{proof}

\remark It is worth noting that the approach developed in this paper should naturally extend to more general P-resolutions. In our work, we have relied heavily on the structure of WW-sequences, where subtraction occurs at exactly two distinct positions. In a more general scenario, where subtraction occurs at $n\geq 2$ distinct positions, a similar approach is expected to be applicable. However, this extension would require a generalized definition of accordion triangulations.

\section{Examples}

In this section, we will apply the algorithm from Corollary \ref{Basic wormhole triangulations algorithm} to classify all basic wormhole triangulations with at most $5$ weights. By Theorem \ref{Alternative proof of HTU theorem on maximal number of extremal P-resolutions}, a basic wormhole triangulation $\mathscr{P}$ has a unique companion. In this section, we denote by $\mathscr{P}^{\vee}$ the unique companion of $\mathscr{P}$.
\begin{example}\label{The simplest basic wormhole triangulations}(Classification of basic wormhole triangulations with $2$ weights). \\
By Definition \ref{Standard family of accordion triangulation of weight n and coherent rotation of diagonals for elements of the standard family of weight n}, the graph $G_{\mathscr{P}^{0}_{x_1,x_2}}$ is the graph in Figure \ref{coherent graph with 2 weights, standard family.}. 
There is only one coherent rotation of diagonals given by the coherent graph $G_{\mathscr{P}^{1}_{x_1,x_2}}$ in Figure \ref{coherent graph with 2 weights, 1-coherent rotation.}.

\begin{figure}[ht]
    \centering
    \hspace*{-2em}
    % First image in a minipage
    \begin{minipage}{0.55\textwidth}  % 45% width of the page
        \centering
        \includegraphics[width=10cm]{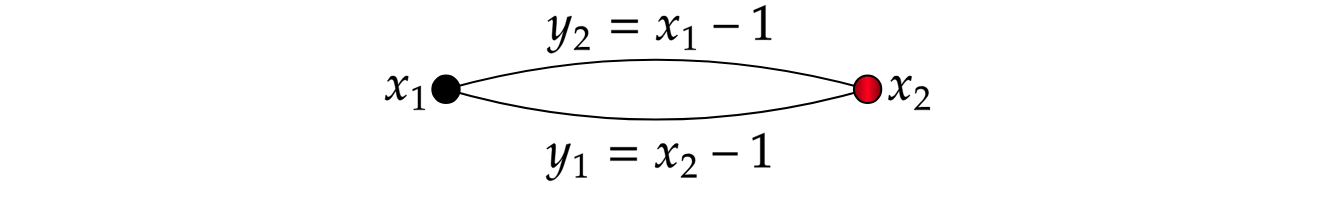}  % Adjust the image path
        \caption{Coherent graph $G_{\mathscr{P}^{0}_{x_1,x_2}}$.}
        \label{coherent graph with 2 weights, standard family.}
    \end{minipage}
    \hfill
    \hspace*{-5em}
    % Second image in another minipage
    \begin{minipage}{0.55\textwidth}
        \centering
        \includegraphics[width=10cm]{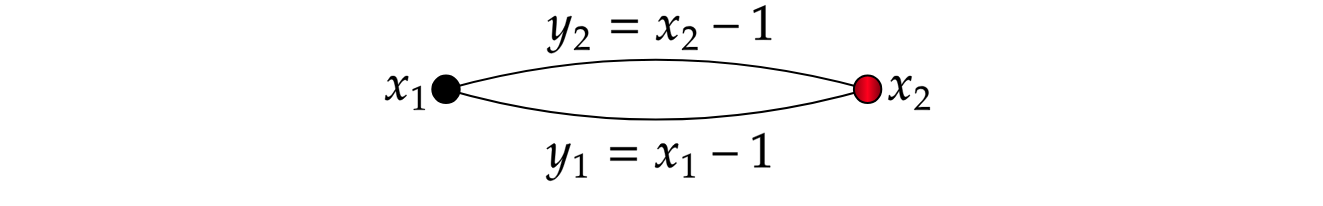}  % Adjust the image path
        \caption{\mbox{Coherent graph $G_{\mathscr{P}^{1}_{x_1,x_2}}$}.}
        \label{coherent graph with 2 weights, 1-coherent rotation.}
    \end{minipage}
\end{figure}

    By Theorem \ref{Main theorem}, we have that $\mathscr{P}^{1}_{x_1,x_2}=(\mathscr{P}^{0}_{x_1,x_2})^\vee$ if and only if $x_1-1=x_2-1$. Thus $(x_1,x_2)=(t,t)$ for $t\geq 3$. 

    \begin{figure}[h]
    \begin{tabular}{ll}
    \includegraphics[width=10cm]{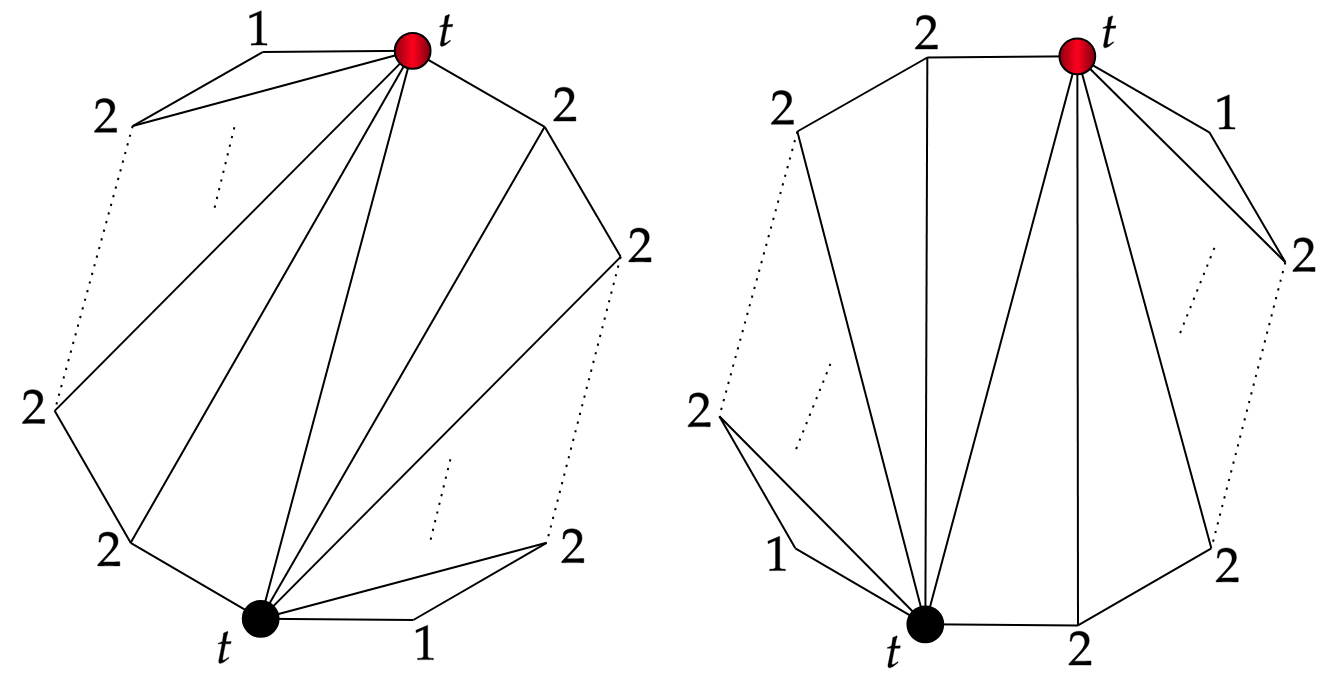}
    \end{tabular}
    \caption{Basic wormhole triangulations with $2$ weights.}
    \label{Wormhole triangulations with $2$ weights.}
    \end{figure}
All basic wormhole triangulations with $2$ weights are either the triangulations shown in Figure \ref{Wormhole triangulations with $2$ weights.} or those obtained by changing the frames in these triangulations, subject to the condition in Step III of Corollary \ref{Basic wormhole triangulations algorithm}.
\end{example}

\begin{example}(Classification of basic wormhole triangulations with $3$ weights).\\
    The graph $G_{\mathscr{P}^{0}_{x_1,x_2,x_3}}$ is the graph in Figure \ref{coherent graph with 3 weights, standard family.}. 
    The coherent rotations of diagonals of $G_{\mathscr{P}^{0}_{x_1,x_2,x_3}}$ are the graphs in Figure \ref{coherent graph with 3 weights, rotation of diagonals.}.

    \begin{figure}[ht]
    \centering
    \hspace*{-6.5em}
    % First image in a minipage
    \begin{minipage}{0.65\textwidth}  % 45% width of the page
        \centering
        \includegraphics[width=9cm]{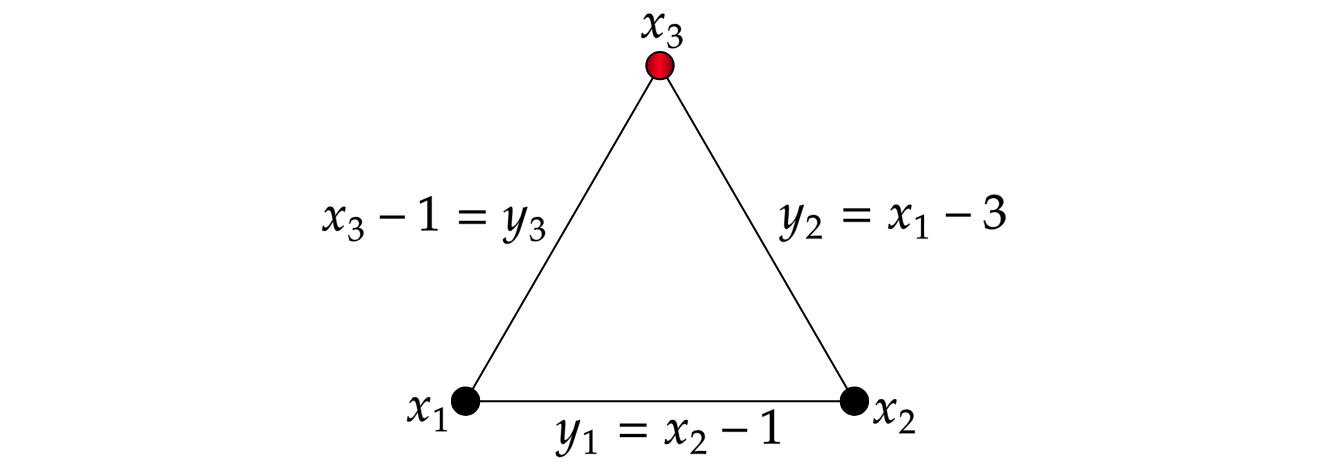}  % Adjust the image path
        \caption{Coherent graph $G_{\mathscr{P}^{0}_{x_1,x_2,x_3}}$.}
        \label{coherent graph with 3 weights, standard family.}
    \end{minipage}
    \hfill
    \hspace*{-7em}
    % Second image in another minipage
    \begin{minipage}{0.65\textwidth}
        \centering
        \includegraphics[width=9cm]{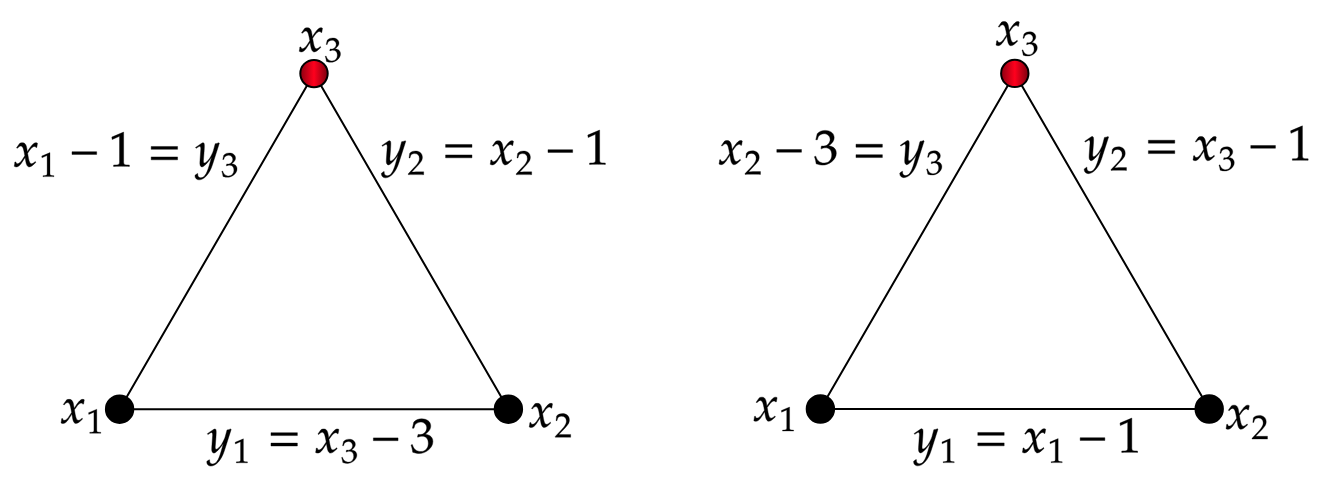}  % Adjust the image path
        \caption{\mbox{Coherent rotations of diagonals.}}
        \label{coherent graph with 3 weights, rotation of diagonals.}
    \end{minipage}
\end{figure}

    \begin{enumerate}
        \item  $\mathscr{P}^{1}_{x_1,x_2,x_3}=(\mathscr{P}^{0}_{x_1,x_2,x_3})^{\vee}$ if and only if we have the relations:
        \begin{align*}
            x_3-1=x_1-1, \hspace{1em} x_2-1=x_3-3, \hspace{1em} x_1-3=x_2-1.
        \end{align*}
        Thus $(x_1,x_2,x_3)=(t,t-2,t)$ for $t\geq 5$. It corresponds to the triangulations in Figure \ref{Basic wormhole triangulations with 3 weights from P0 and P1.}.
    
        \begin{figure}[h]
        \begin{tabular}{ll}
        \hspace*{-5em}
        \includegraphics[width=10cm]{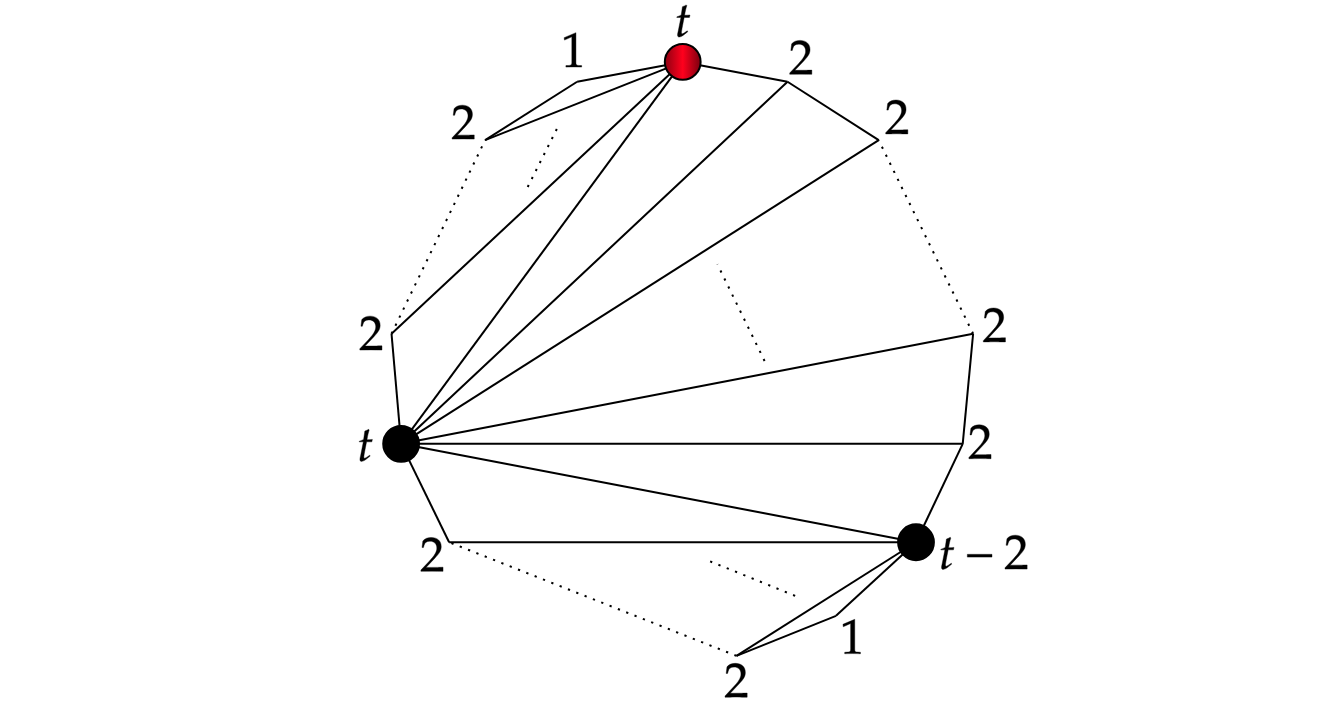}
        & \hspace*{-7em}
        \includegraphics[width=10cm]{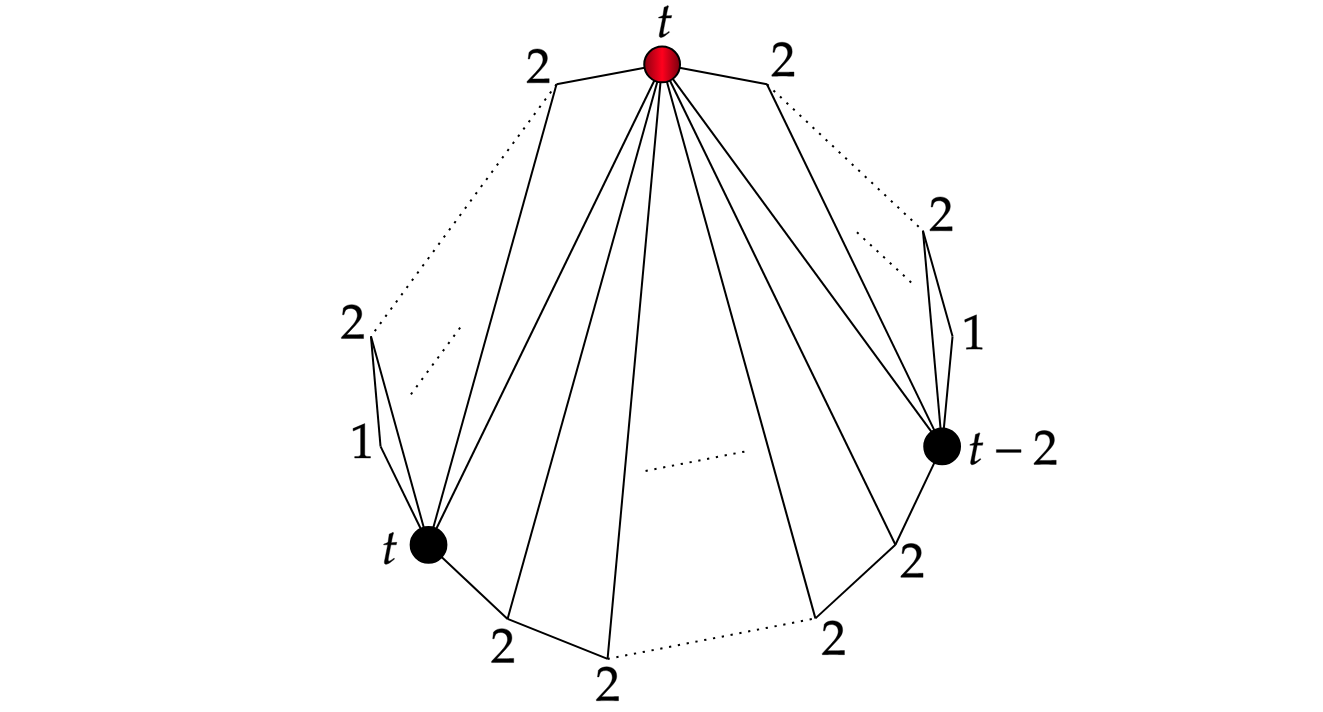}
        \end{tabular}
        \caption{Basic wormhole triangulations with $3$ weights from $\mathscr{P}_0$ and $\mathscr{P}_1$.}
        \label{Basic wormhole triangulations with 3 weights from P0 and P1.}
        \end{figure}
        
        \item $\mathscr{P}^{2}_{x_1,x_2,x_3}=(\mathscr{P}^{0}_{x_1,x_2,x_3})^{\vee}$ if and only if we have the relations:
        \begin{align*}
            x_3-1=x_2-3, \hspace{1em} x_2-1=x_1-1, \hspace{1em} x_1-3=x_3-1.
        \end{align*}
        Thus $(x_1,x_2,x_3)=(t,t,t-2)$ for  $t\geq 5$. It corresponds to the triangulations in Figure \ref{Basic wormhole triangulations with 3 weights from P0 and P2.}.
        \begin{figure}[h]
        \begin{tabular}{ll}
        \hspace*{-5em}
        \includegraphics[width=10cm]{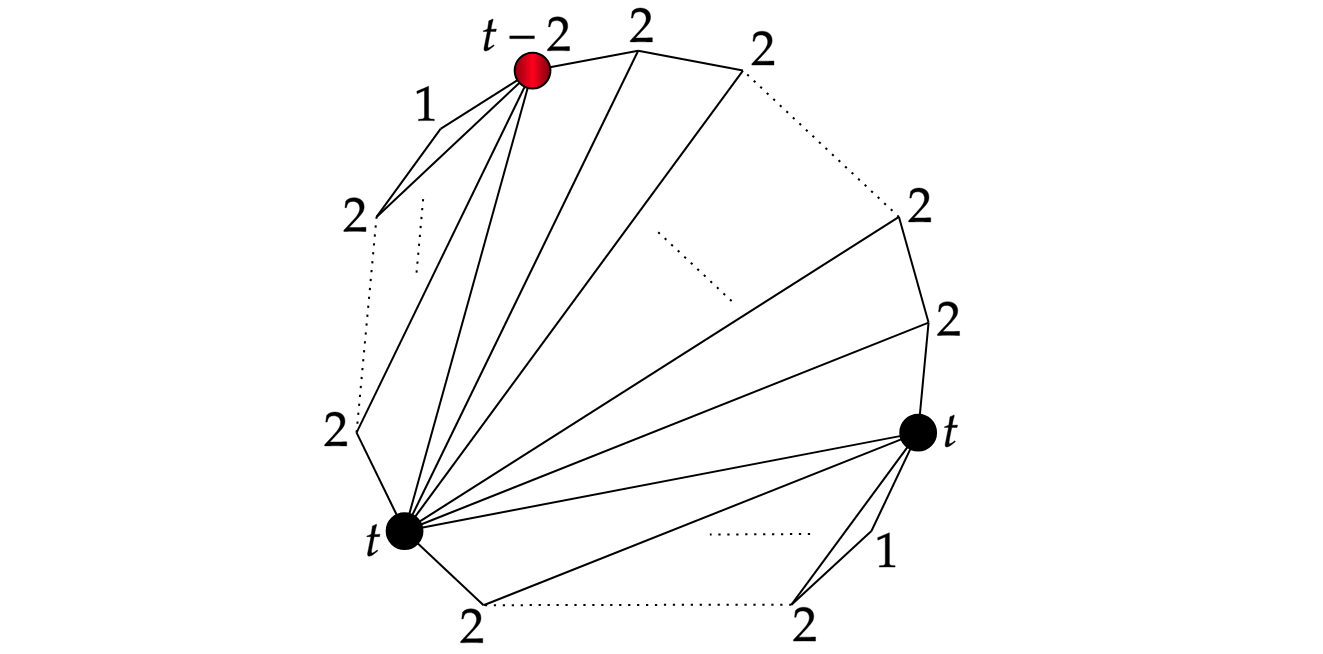}
        & \hspace*{-7em}
        \includegraphics[width=10cm]{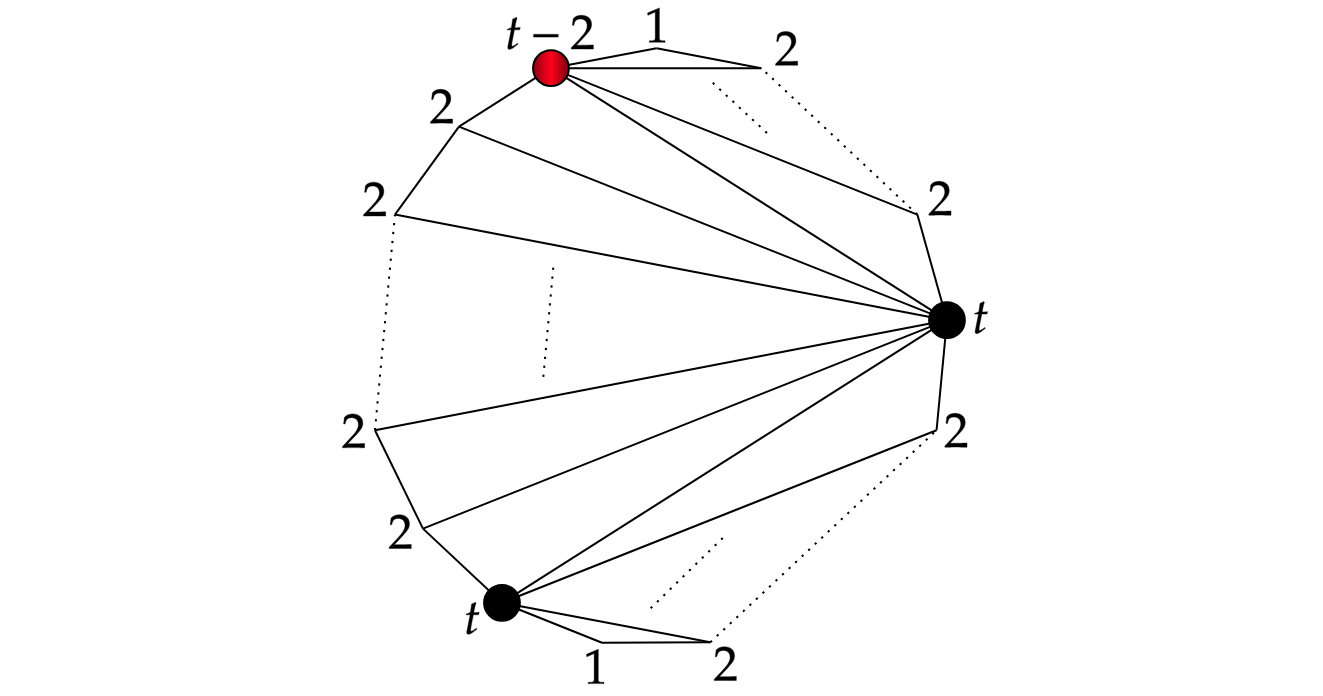}
        \end{tabular}
        \caption{Basic wormhole triangulations with $3$ weights from $\mathscr{P}_0$ and $\mathscr{P}_2$.}
        \label{Basic wormhole triangulations with 3 weights from P0 and P2.}
        \end{figure}
    \end{enumerate}
All basic wormhole triangulations with $3$ weights are either the triangulations shown in Figures \ref{Basic wormhole triangulations with 3 weights from P0 and P1.},\ref{Basic wormhole triangulations with 3 weights from P0 and P2.}, or those obtained by changing the frames in these triangulations, subject to the condition in Step III of Corollary \ref{Basic wormhole triangulations algorithm}.
\end{example}

\begin{example}(Classification of basic wormhole triangulations with $4$ weights).\\
    The graph $G_{\mathscr{P}^{0}_{x_1,x_2,x_3,x_4}}$ is the graph in Figure \ref{coherent graph with $4$ weights.}. The coherent rotations of diagonals of $G_{\mathscr{P}^{0}_{x_1,x_2,x_3,x_4}}$ are the graphs in Figure \ref{coherent graph with 4 weights, rotation of diagonals.}.
    
    \begin{figure}[h]
    \begin{tabular}{ll}
    \includegraphics[width=8cm]{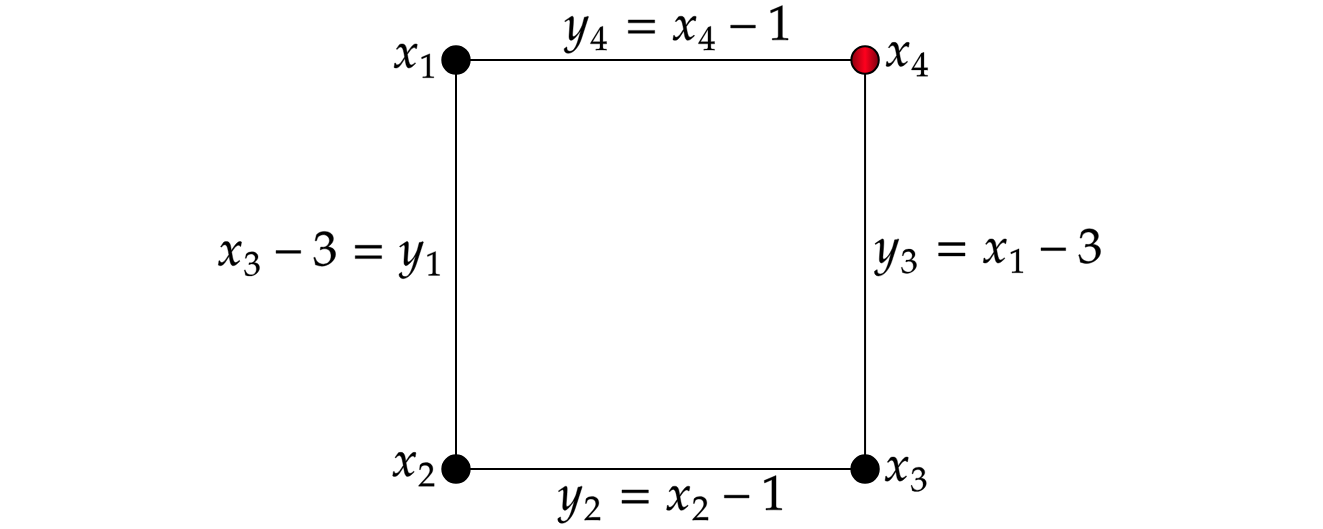}
    \end{tabular}
    \caption{Coherent graph $G_{\mathscr{P}^{0}_{x_1,x_2,x_3,x_4}}$.}
    \label{coherent graph with $4$ weights.}
    \end{figure}

    \begin{figure}[h]
    \begin{tabular}{ll}
    \includegraphics[width=10cm]{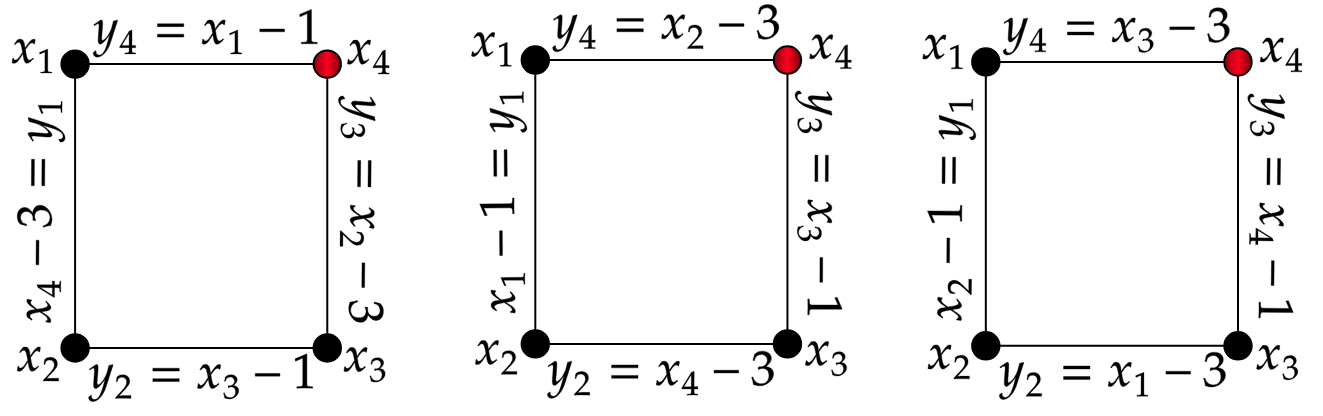}
    \end{tabular}
    \caption{Coherent rotations of diagonals.}
    \label{coherent graph with 4 weights, rotation of diagonals.}
    \end{figure}

\begin{enumerate}
    \item $\mathscr{P}^{1}_{x_1,x_2,x_3,x_4}=(\mathscr{P}^{0}_{x_1,x_2,x_3,x_4})^{\vee}$ if and only if we have the relations:
    $$x_3-3=x_4-3, \hspace{1em}x_2-1=x_3-1, \hspace{1em} x_1-3=x_2-3, \hspace{1em} x_4-1=x_1-1.$$
    Thus $(x_1,x_2,x_3,x_4)=(t,t,t,t)$ for $t\geq 3$. It corresponds to the triangulations in Figure \ref{Basic wormhole triangulations with 4 weights from P0 and P1.}.
    \begin{figure}[h]
        \begin{tabular}{ll}
        \includegraphics[width=10cm]{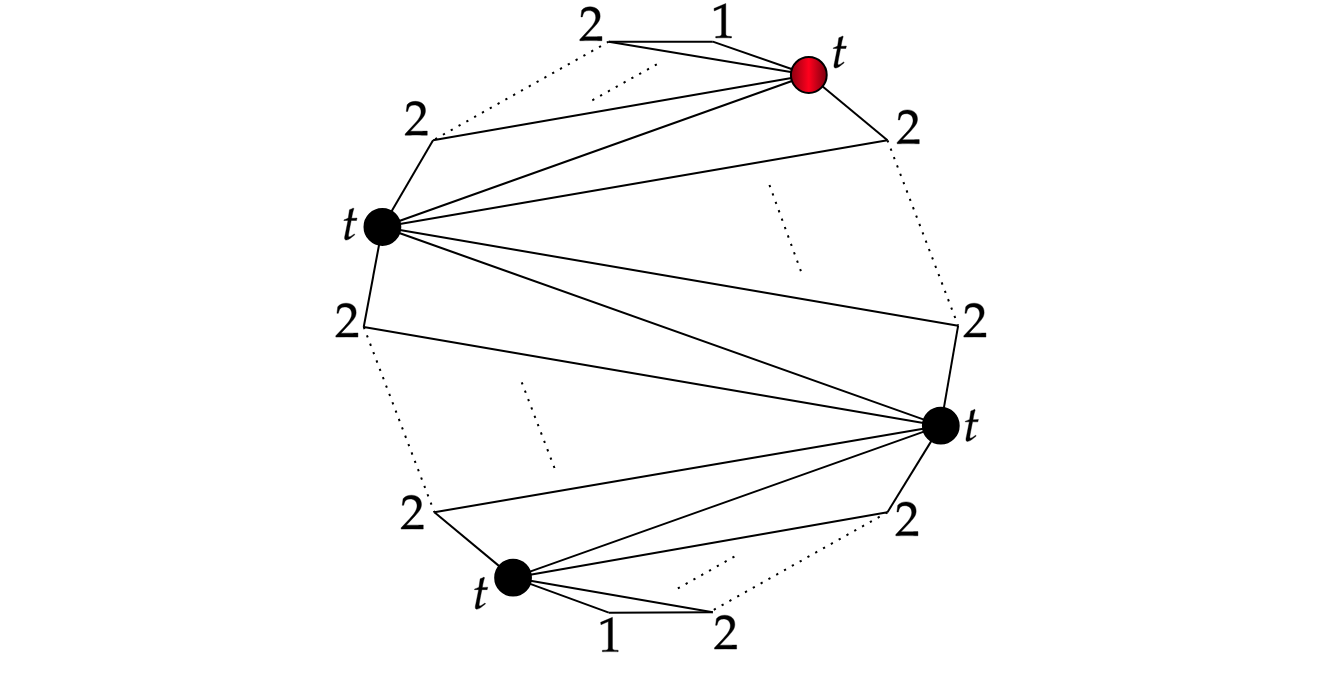}
        & \hspace*{-7em}
        \includegraphics[width=10cm]{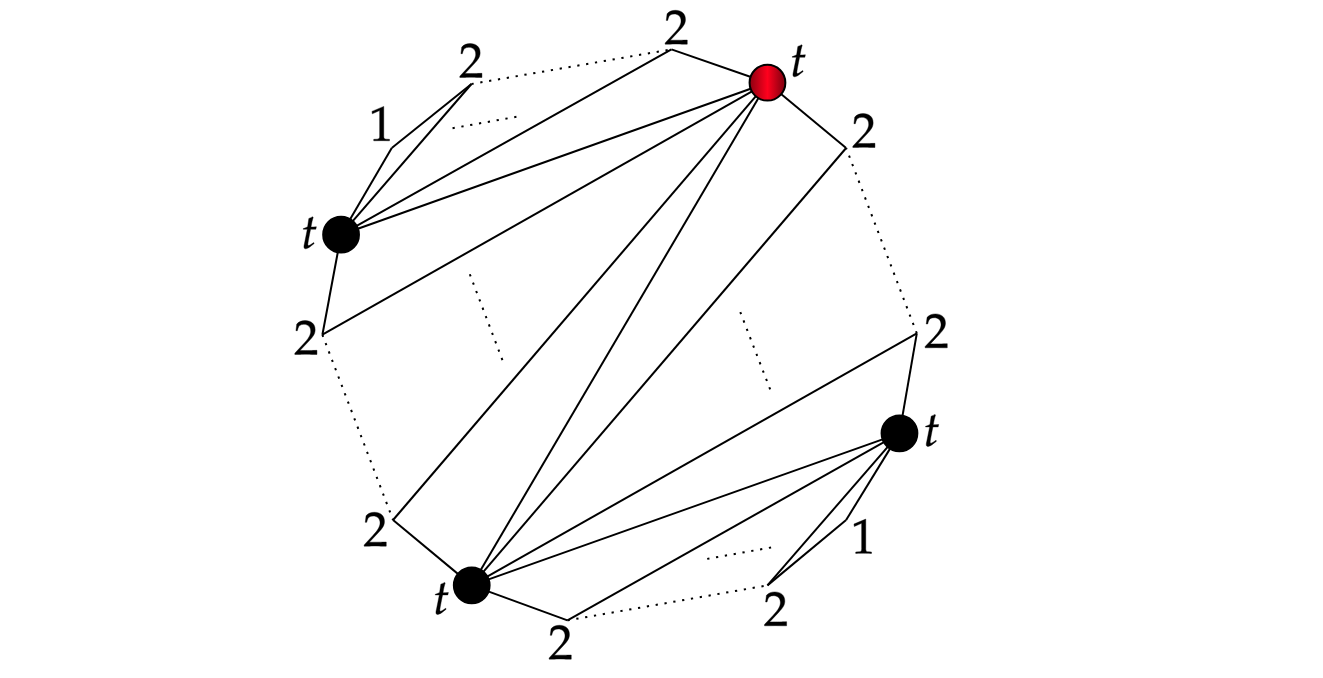}
        \end{tabular}
        \caption{Basic wormhole triangulations with $4$ weights from $\mathscr{P}_0$ and $\mathscr{P}_1$.}
        \label{Basic wormhole triangulations with 4 weights from P0 and P1.}
    \end{figure}
    \item $\mathscr{P}^{2}_{x_1,x_2,x_3,x_4}=(\mathscr{P}^{0}_{x_1,x_2,x_3,x_4})^{\vee}$ if and only if we have the relations:
    $$x_3-3=x_1-1, \hspace{1em} x_2-1=x_4-3, \hspace{1em} x_1-3=x_3-1, \hspace{1em} x_4-1=x_2-3.$$
    Since $x_3=x_1+2$ and $x_1=x_3+2$, this system is inconsistent. Indeed, by Theorem \ref{Main theorem}, we can conclude that this case ($n=4$ and $m=2$) does not produce a pair of basic wormhole triangulations because $\text{gcd}(n,m) \nmid n-\lfloor \frac{m}{2}\rfloor$ and $\text{gcd}(n,m)\nmid \frac{n}{2}-\lfloor \frac{m}{2}\rfloor$.
    \item $\mathscr{P}^{3}_{x_1,x_2,x_3,x_4}=(\mathscr{P}^{0}_{x_1,x_2,x_3,x_4})^{\vee}$ if and only if we have the relations:
    $$x_3-3=x_2-1, \hspace{1em} x_2-1=x_1-3, \hspace{1em} x_1-3=x_4-1, \hspace{1em} x_4-1=x_3-3.$$
    Thus $(x_1,x_2,x_3,x_4)=(t,t-2,t,t-2)$ for $t\geq 5$. It corresponds to the triangulations in Figure \ref{Basic wormhole triangulations with 4 weights from P0 and P3.}.
    \begin{figure}[h]
        \hspace*{-4em}
        \begin{tabular}{ll}
        \includegraphics[width=10cm]{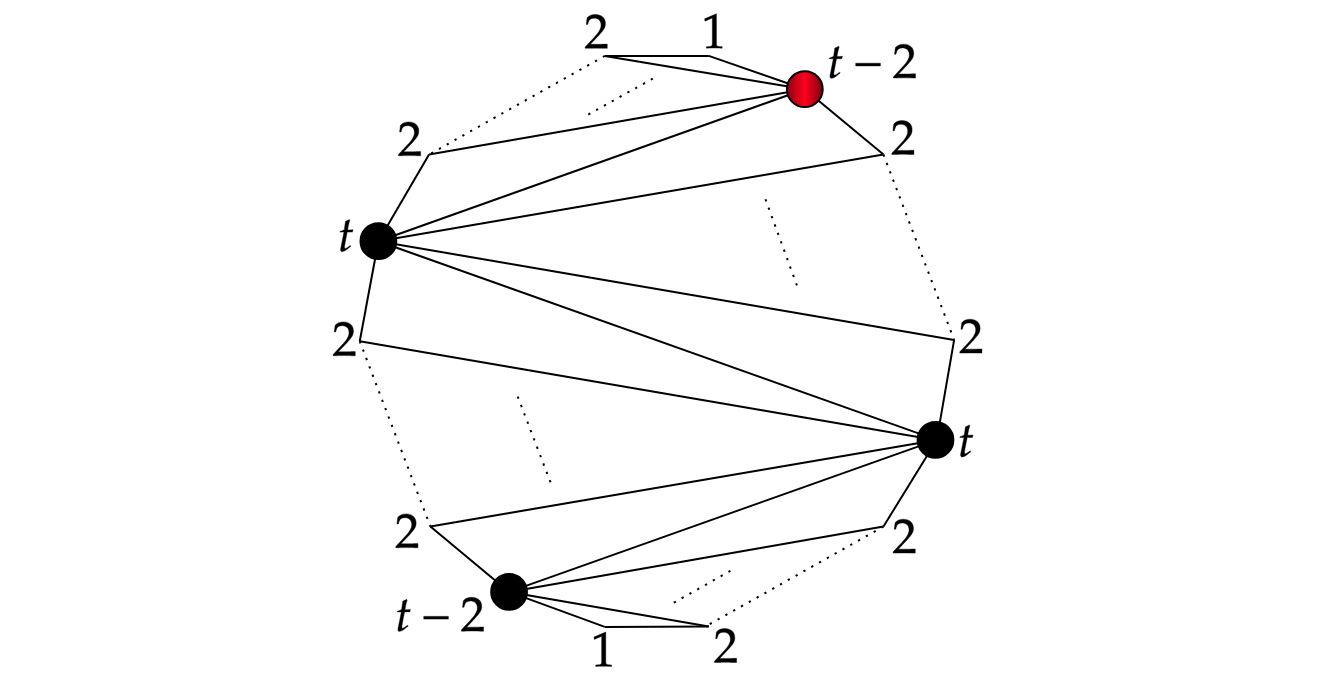}
        & \hspace*{-7em}
        \includegraphics[width=10cm]{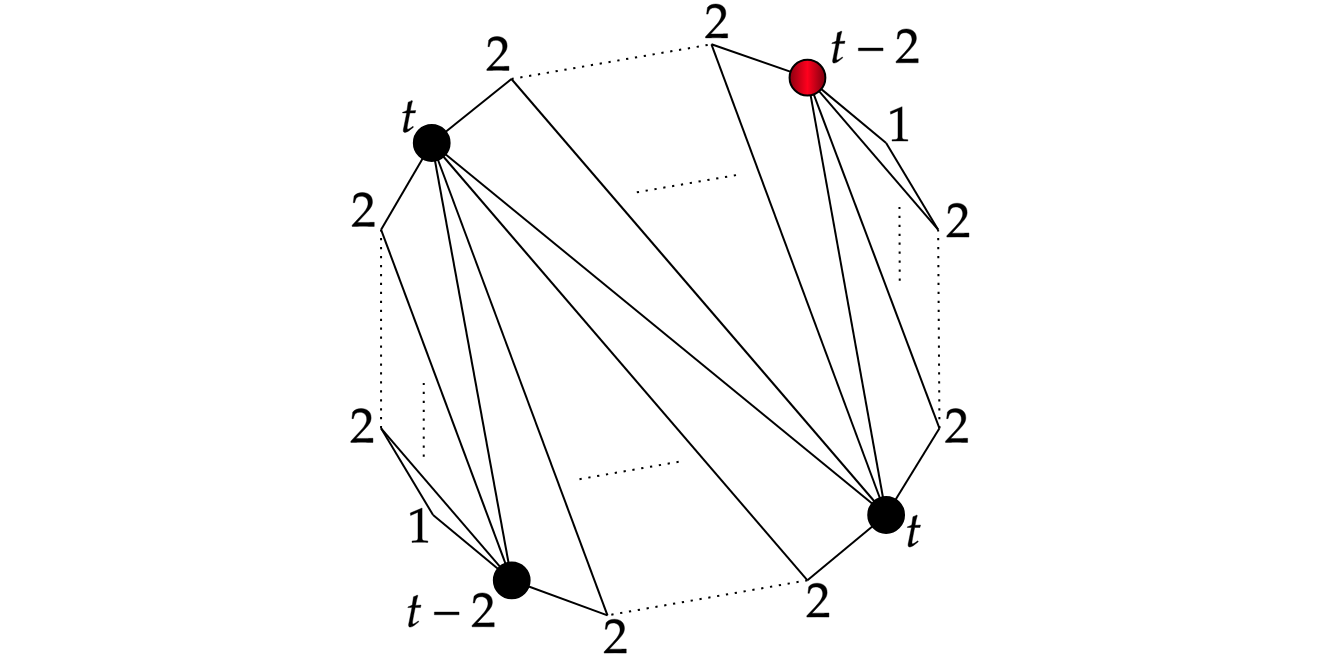}
        \end{tabular}
        \caption{Basic wormhole triangulations with $4$ weights from $\mathscr{P}_0$ and $\mathscr{P}_3$.}
        \label{Basic wormhole triangulations with 4 weights from P0 and P3.}
    \end{figure}
\end{enumerate}
All basic wormhole triangulations with $4$ weights are either the triangulations shown in Figures \ref{Basic wormhole triangulations with 4 weights from P0 and P1.},\ref{Basic wormhole triangulations with 4 weights from P0 and P3.}, or those obtained by changing the frames in these triangulations, subject to the condition in Step III of Corollary \ref{Basic wormhole triangulations algorithm}.
\end{example}

\begin{example}
    (Classification of basic wormhole triangulations with $5$ weights).\\  
    The graph $G_{\mathscr{P}^{0}_{x_1,x_2,x_3,x_4,x_5}}$ is the graph in Figure \ref{coherent graph with $5$ weights.}.
    \begin{figure}[h]
    \begin{tabular}{ll}
    \includegraphics[width=9cm]{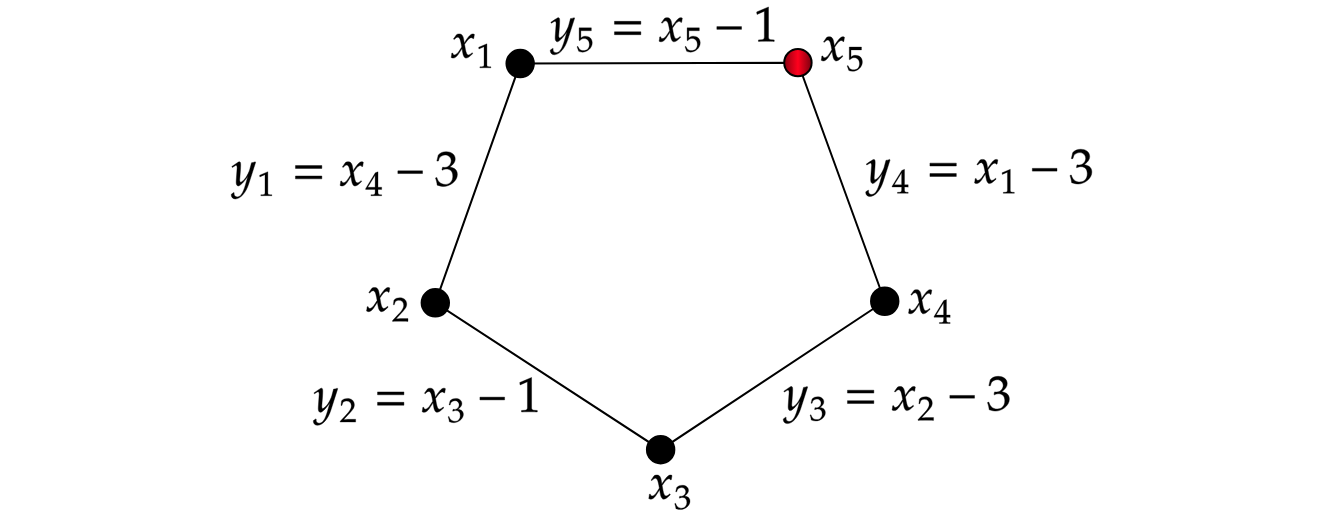}
    \end{tabular}
    \caption{Coherent graph $G_{\mathscr{P}^{0}_{x_1,x_2,x_3,x_4,x_5 }}$.}
    \label{coherent graph with $5$ weights.}
    \end{figure}
    \begin{enumerate}
        \item $\mathscr{P}^{1}_{x_1,x_2,x_3,x_4,x_5}=(\mathscr{P}^{0}_{x_1,x_2,x_3,x_4,x_5})^{\vee}$ if and only if we have the relations:
        $$x_4-3=x_5-3,\hspace{1em} x_3-1=x_4-3, \hspace{1em} x_2-3=x_3-1,\hspace{1em} x_1-3=x_2-3, \hspace{1em} x_5-1=x_1-1.$$
        Thus $(x_1,x_2,x_3,x_4,x_5)=(t,t,t-2,t,t)$ for $t\geq 5$. It corresponds to the triangulations in Figure \ref{Basic wormhole triangulations with 5 weights from P0 and P1.}.
        \begin{figure}[h]
        \hspace*{-5em}
        \begin{tabular}{ll}
        \includegraphics[width=10cm]{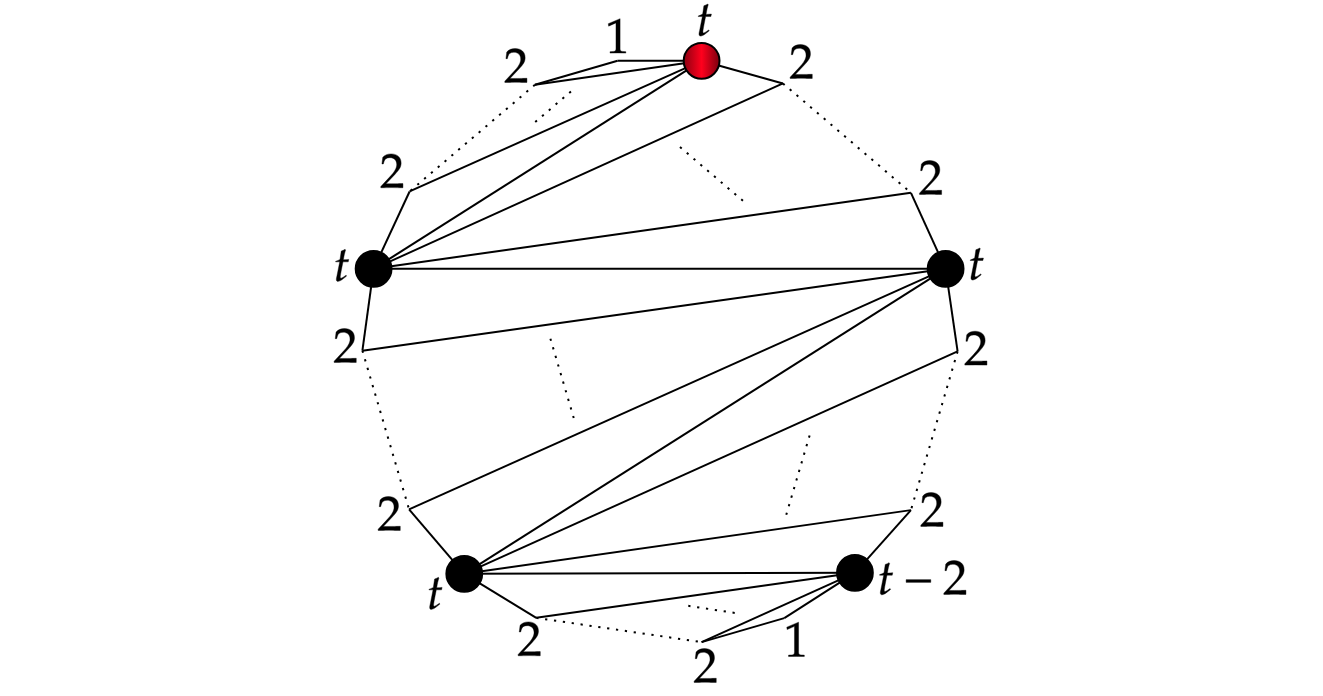}
        & \hspace*{-6em}
        \includegraphics[width=10cm]{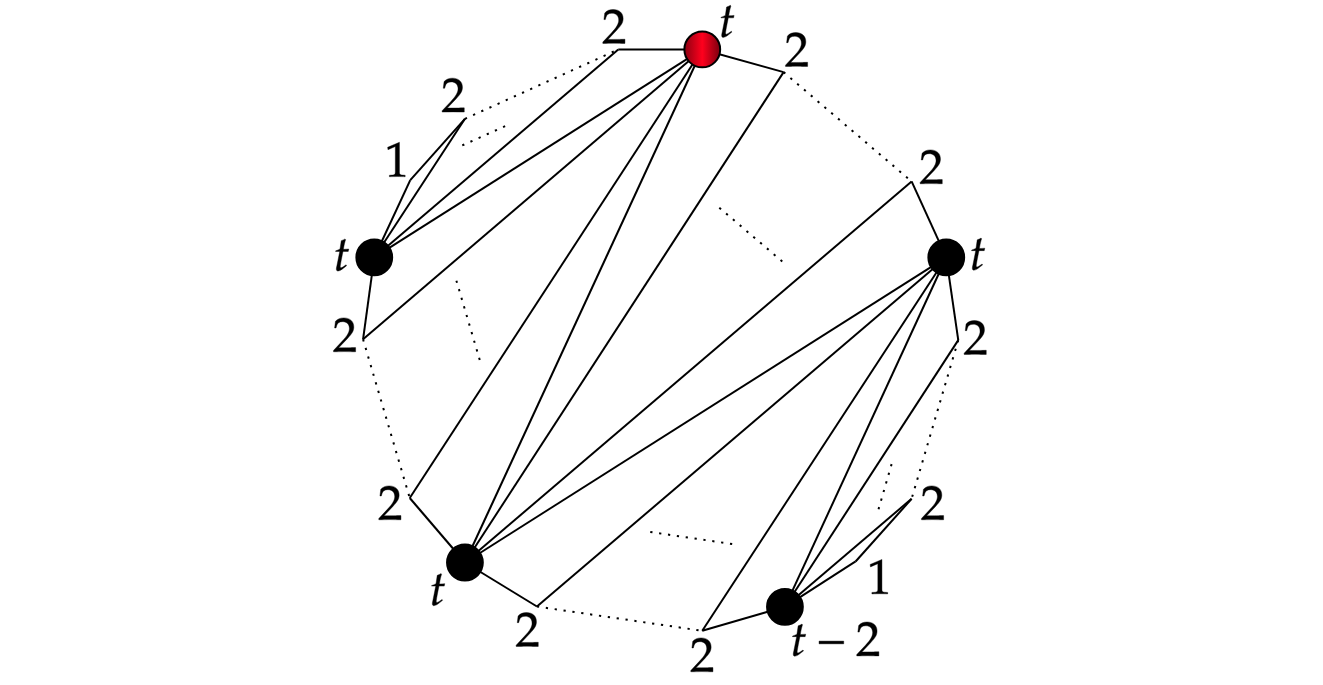}
        \end{tabular}
        \caption{Basic wormhole triangulations with $5$ weights from $\mathscr{P}_0$ and $\mathscr{P}_1$.}
        \label{Basic wormhole triangulations with 5 weights from P0 and P1.}
        \end{figure}

        \item $\mathscr{P}^{2}_{x_1,x_2,x_3,x_4,x_5}=(\mathscr{P}^{0}_{x_1,x_2,x_3,x_4,x_5})^{\vee}$ if and only if we have the relations:
        $$x_4-3=x_1-1,\hspace{1em} x_3-1=x_5-3, \hspace{1em} x_2-3=x_4-1,\hspace{1em} x_1-3=x_3-3, \hspace{1em} x_5-1=x_2-3.$$
        Thus $(x_1,x_2,x_3,x_4,x_5)=(t,t+4,t,t+2,t+2)$ for $t\geq 3$. It corresponds to the triangulations in Figure \ref{Basic wormhole triangulations with 5 weights from P0 and P2.}.

        \begin{figure}[h]
        \hspace*{-5em}
        \begin{tabular}{ll}
        \includegraphics[width=10cm]{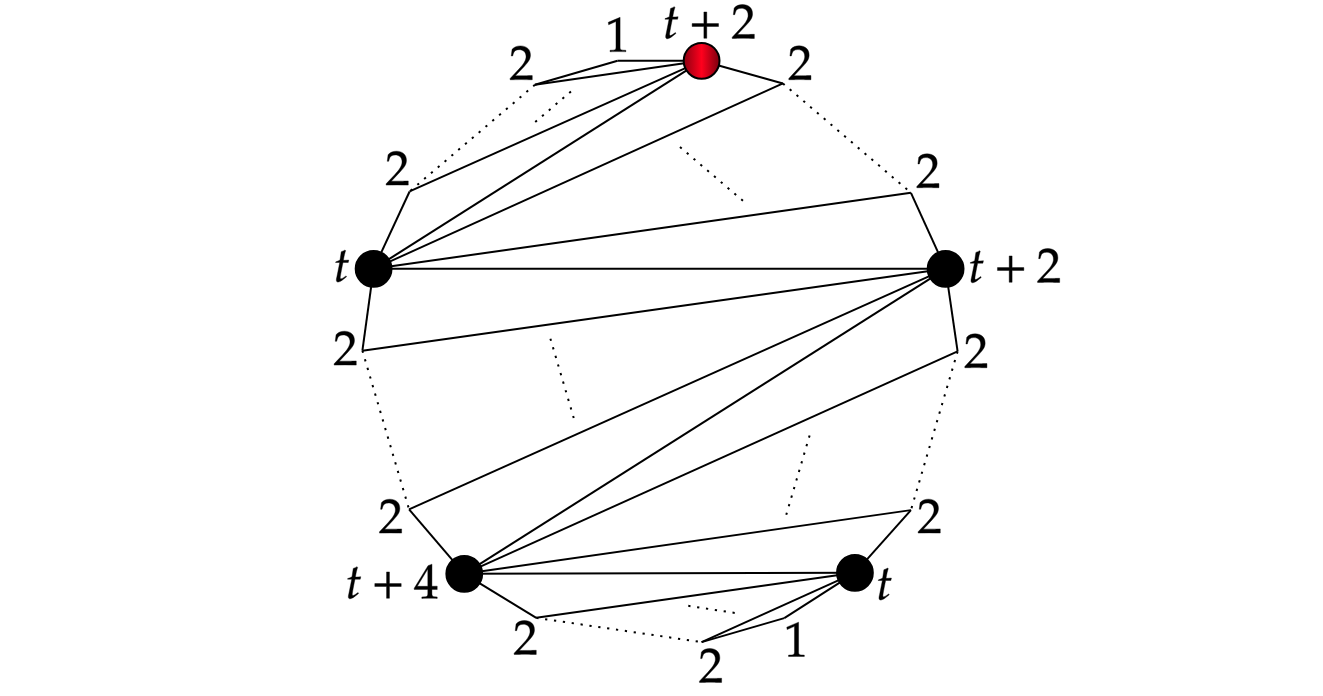}
        & \hspace*{-6em}
        \includegraphics[width=10cm]{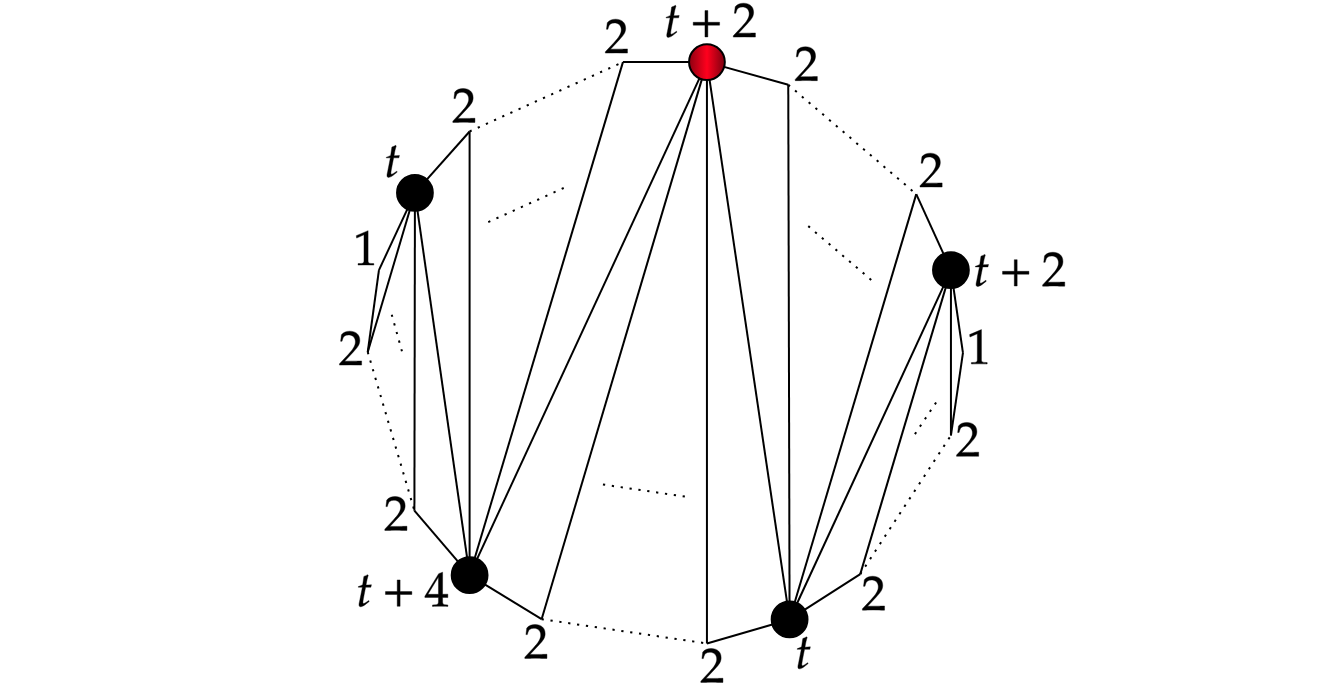}
        \end{tabular}
        \caption{Basic wormhole triangulations with $5$ weights from $\mathscr{P}_0$ and $\mathscr{P}_2$.}
        \label{Basic wormhole triangulations with 5 weights from P0 and P2.}
        \end{figure}
        
        \item $\mathscr{P}^{3}_{x_1,x_2,x_3,x_4,x_5}=(\mathscr{P}^{0}_{x_1,x_2,x_3,x_4,x_5})^{\vee}$ if and only if we have the relations:
        $$x_4-3=x_2-1,\hspace{1em} x_3-1=x_1-3, \hspace{1em} x_2-3=x_5-3,\hspace{1em} x_1-3=x_4-1, \hspace{1em} x_5-1=x_3-3.$$
        Thus $(x_1,x_2,x_3,x_4,x_5)=(t,t-4,t-2,t-2,t-4)$ for $t\geq 7$. It corresponds to the triangulations in Figure \ref{Basic wormhole triangulations with 5 weights from P0 and P3.}.
        \begin{figure}[h]
        \hspace*{-5em}
        \begin{tabular}{ll}
        \includegraphics[width=10cm]{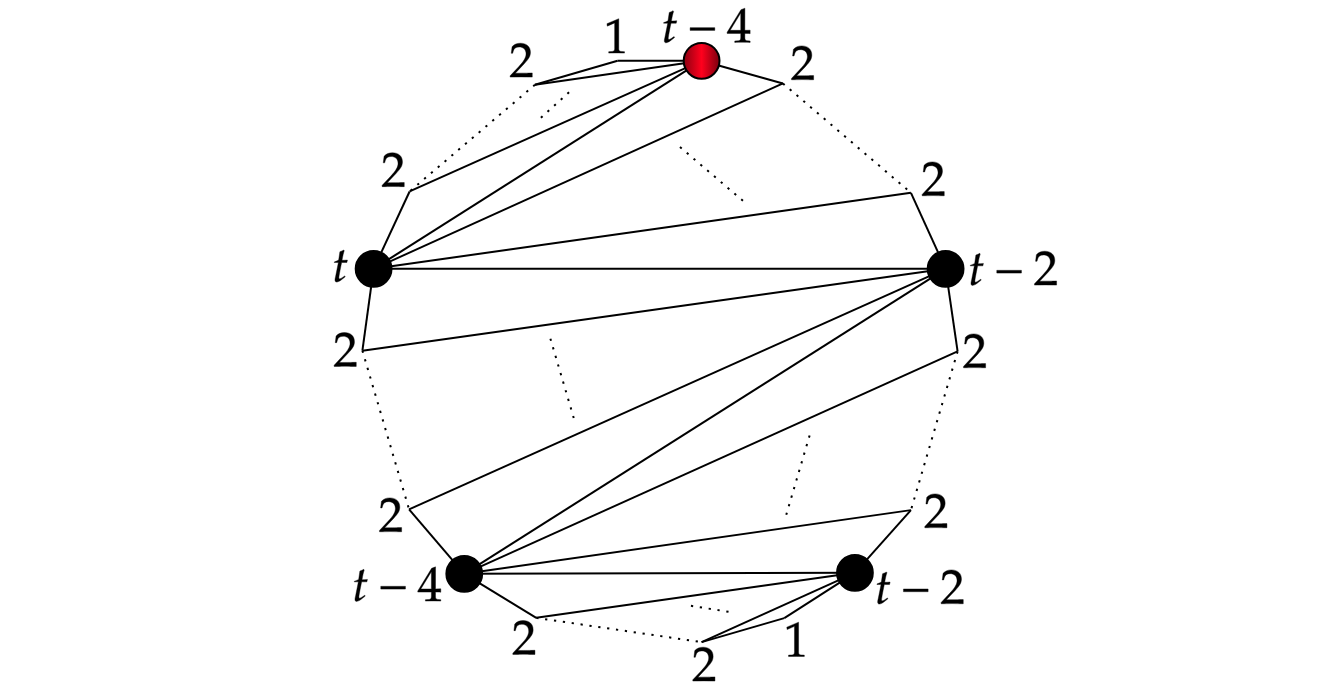}
        & \hspace*{-6em}
        \includegraphics[width=10cm]{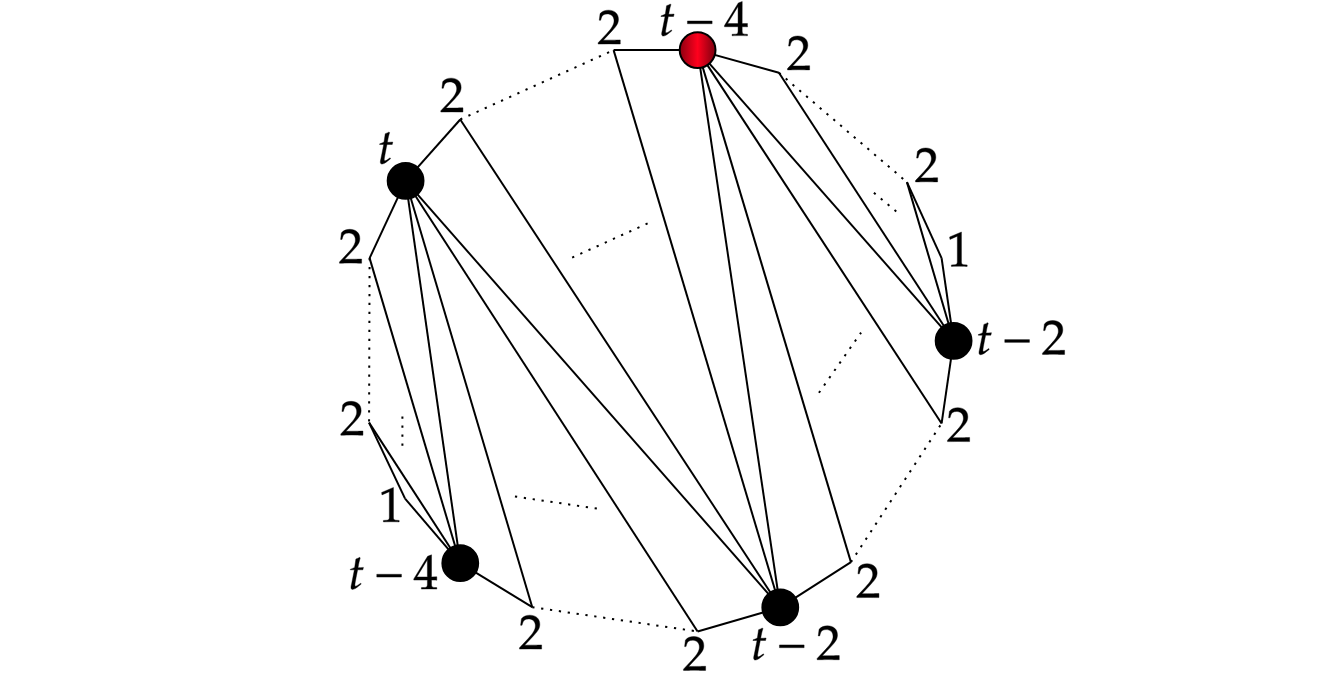}
        \end{tabular}
        \caption{Basic wormhole triangulations with $5$ weights from $\mathscr{P}_0$ and $\mathscr{P}_3$.}
        \label{Basic wormhole triangulations with 5 weights from P0 and P3.}
        \end{figure}

        \item $\mathscr{P}^{4}_{x_1,x_2,x_3,x_4,x_5}=(\mathscr{P}^{0}_{x_1,x_2,x_3,x_4,x_5})^{\vee}$ if and only if we have the relations:
        $$x_4-3=x_3-3,\hspace{1em} x_3-1=x_2-1, \hspace{1em} x_2-3=x_1-3,\hspace{1em} x_1-3=x_5-1, \hspace{1em} x_5-1=x_4-3.$$
        Thus $(x_1,x_2,x_3,x_4,x_5)=(t,t,t,t,t-2)$ for $t\geq 5$. It corresponds to the triangulations in Figure \ref{Basic wormhole triangulations with 5 weights from P0 and P4.}.
        \begin{figure}[h]
        \begin{tabular}{ll}
        \hspace*{-5em}
        \includegraphics[width=10cm]{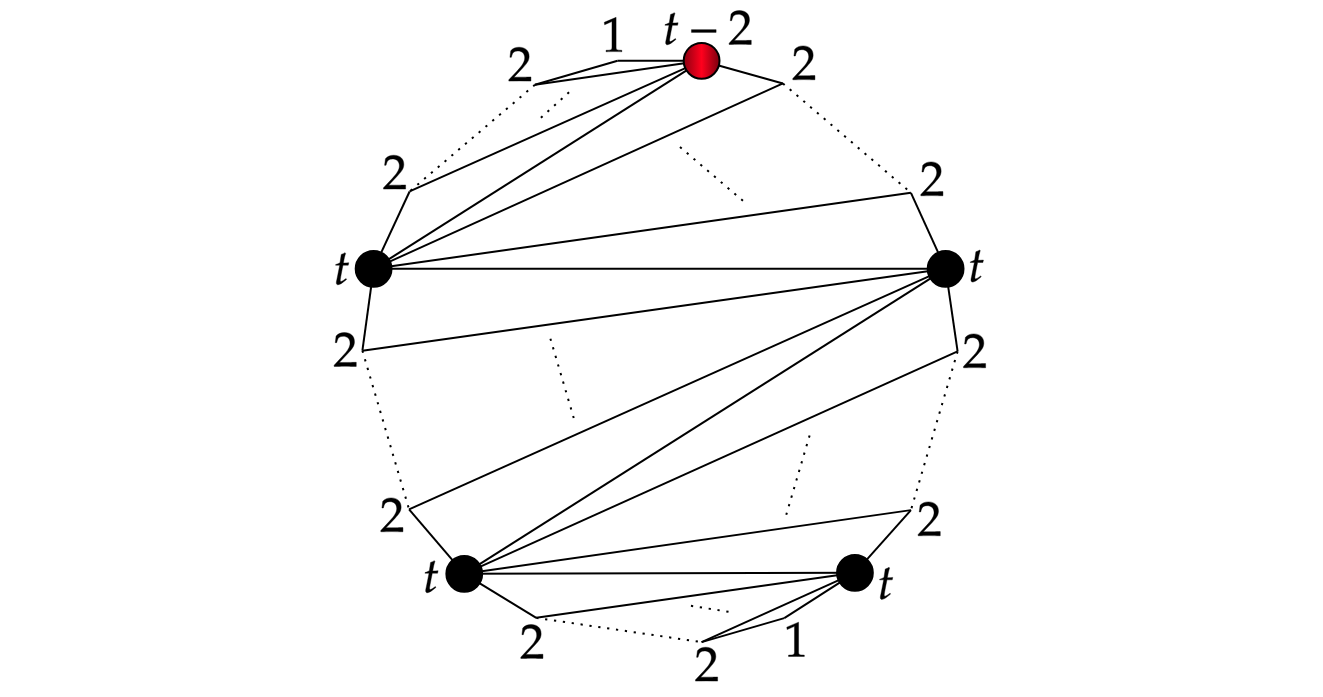}
        & \hspace*{-6em}
        \includegraphics[width=10cm]{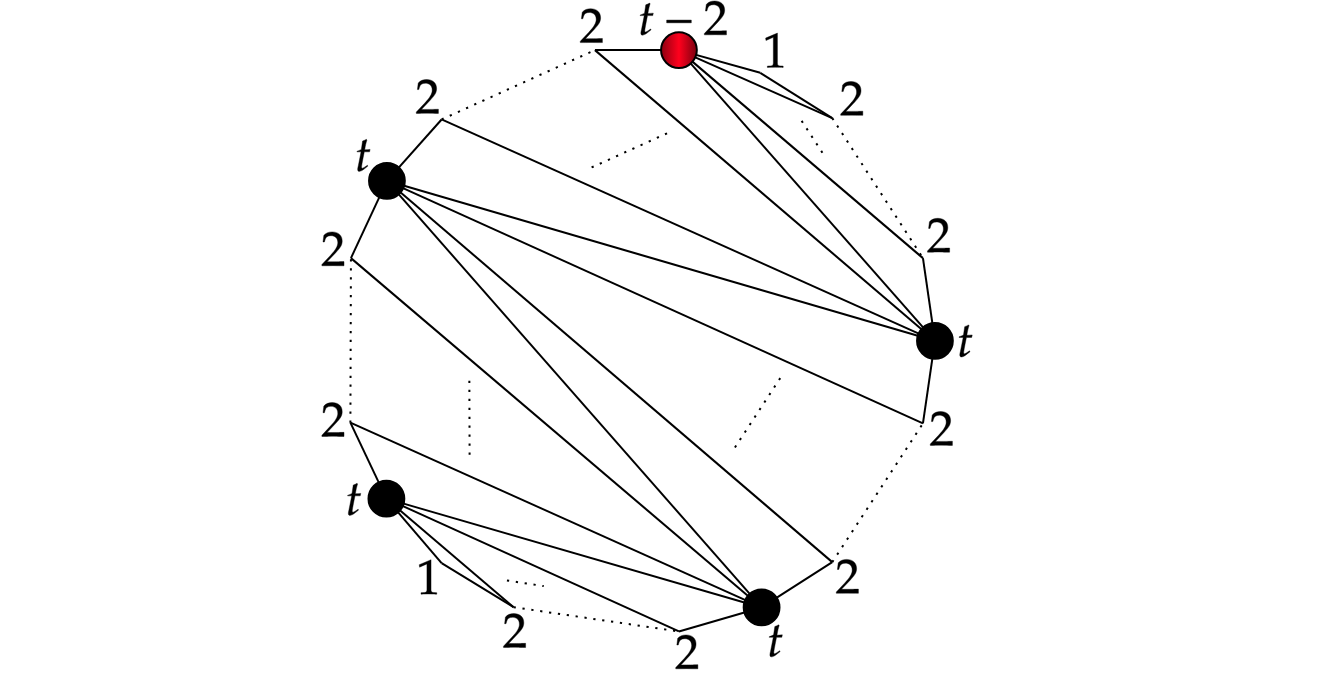}
        \end{tabular}
        \caption{Basic wormhole triangulations with $5$ weights from $\mathscr{P}_0$ and $\mathscr{P}_4$.}
        \label{Basic wormhole triangulations with 5 weights from P0 and P4.}
        \end{figure}
    \end{enumerate}
All basic wormhole triangulations with $5$ weights are either the triangulations shown in Figures \ref{Basic wormhole triangulations with 5 weights from P0 and P1.},\ref{Basic wormhole triangulations with 5 weights from P0 and P2.},\ref{Basic wormhole triangulations with 5 weights from P0 and P3.},\ref{Basic wormhole triangulations with 5 weights from P0 and P4.}, or those obtained by changing the frames in these triangulations, subject to the condition in Step III of Corollary \ref{Basic wormhole triangulations algorithm}.
\end{example}

\begin{example}
    (Recovering basic wormhole singularities from a basic wormhole triangulation). Consider the triangulations in Figure \ref{Basic wormhole triangulations with 5 weights from P0 and P3.} for $t=7$.
    \begin{figure}[H] %
        \begin{tabular}{ll}
        \hspace*{-5em}
        \includegraphics[width=10cm]{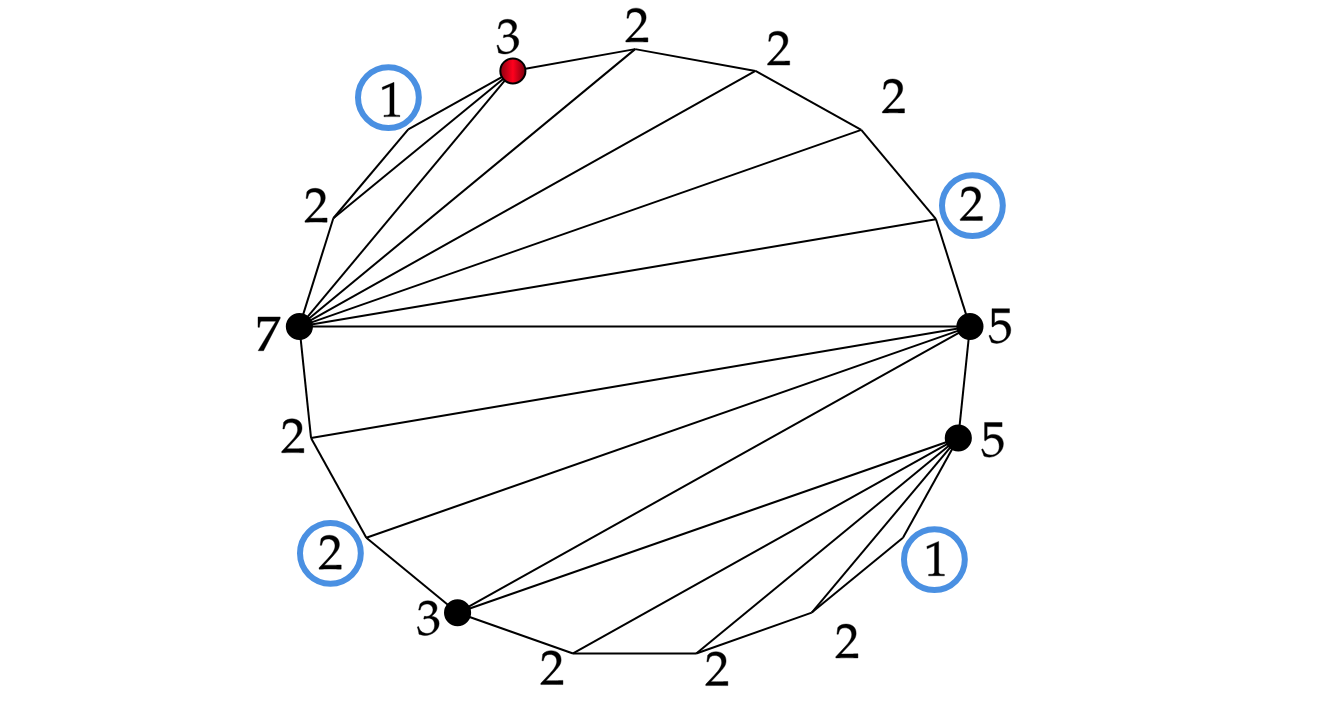}
        & \hspace*{-6em}
        \includegraphics[width=10cm]{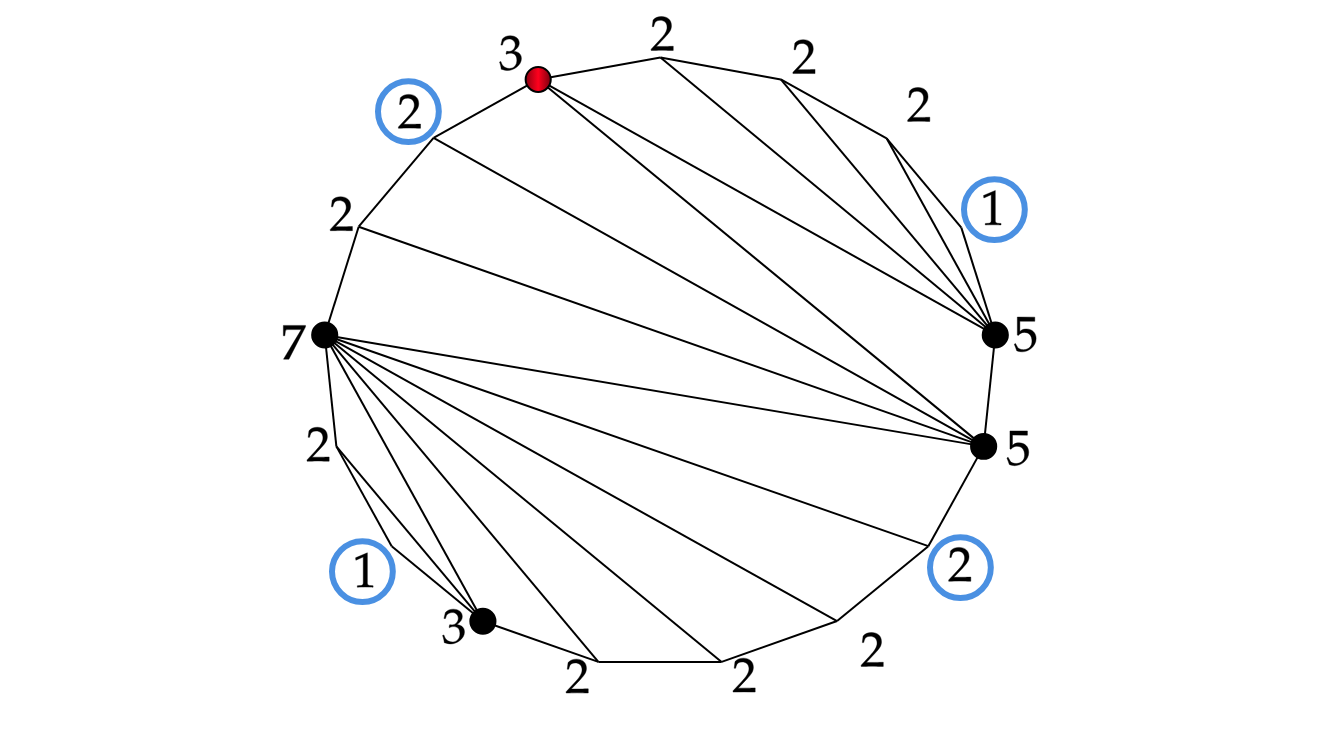}
        \end{tabular}
        \caption{Basic wormhole triangulations in Figure \ref{Basic wormhole triangulations with 5 weights from P0 and P3.} for $t=7$.}
        \label{final example}
    \end{figure}
    In this case, the extended zero chains are
    $$0=[1,2,7,2,2,3,2,2,2,1,5,5,2,2,2,2\mid 3],$$
    $$0=[2,2,7,2,1,3,2,2,2,2,5,5,1,2,2,2\mid 3].$$
    These are WW-decompositions of the H-J continued fraction
    $$\frac{m_0}{m_0-a_0}:=\frac{31901}{21901}=[2,2,7,2,2,3,2,2,2,2,5,5,2,2,2,2].$$
    It defines the wormhole singularity $\frac{1}{31901}(1,10000)$. Its minimal resolutions is given by the H-J continued fraction
    $$\frac{31901}{10000}=[4,2,2,2,2,5,7,2,2,3,2,2,6].$$
    We can change the frame from these two accordion triangulations via cyclic permutation (the only restriction is that the new hidden index cannot be assigned to positions where a vertex with index $1$ is located in any of these triangulations), see Step III of Corollary \ref{Basic wormhole triangulations algorithm}.
    \begin{figure}[H] %
        \begin{tabular}{ll}
        \hspace*{-5em}
        \includegraphics[width=10cm]{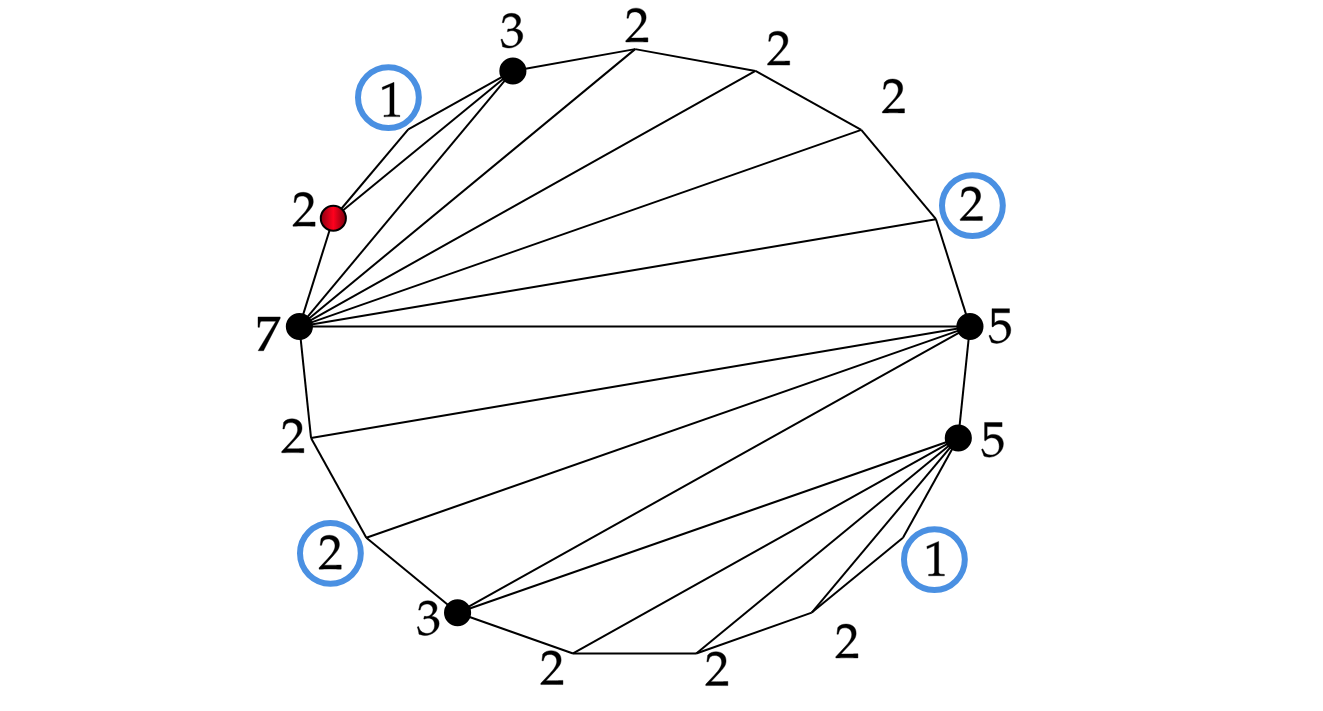}
        & \hspace*{-6em}
        \includegraphics[width=10cm]{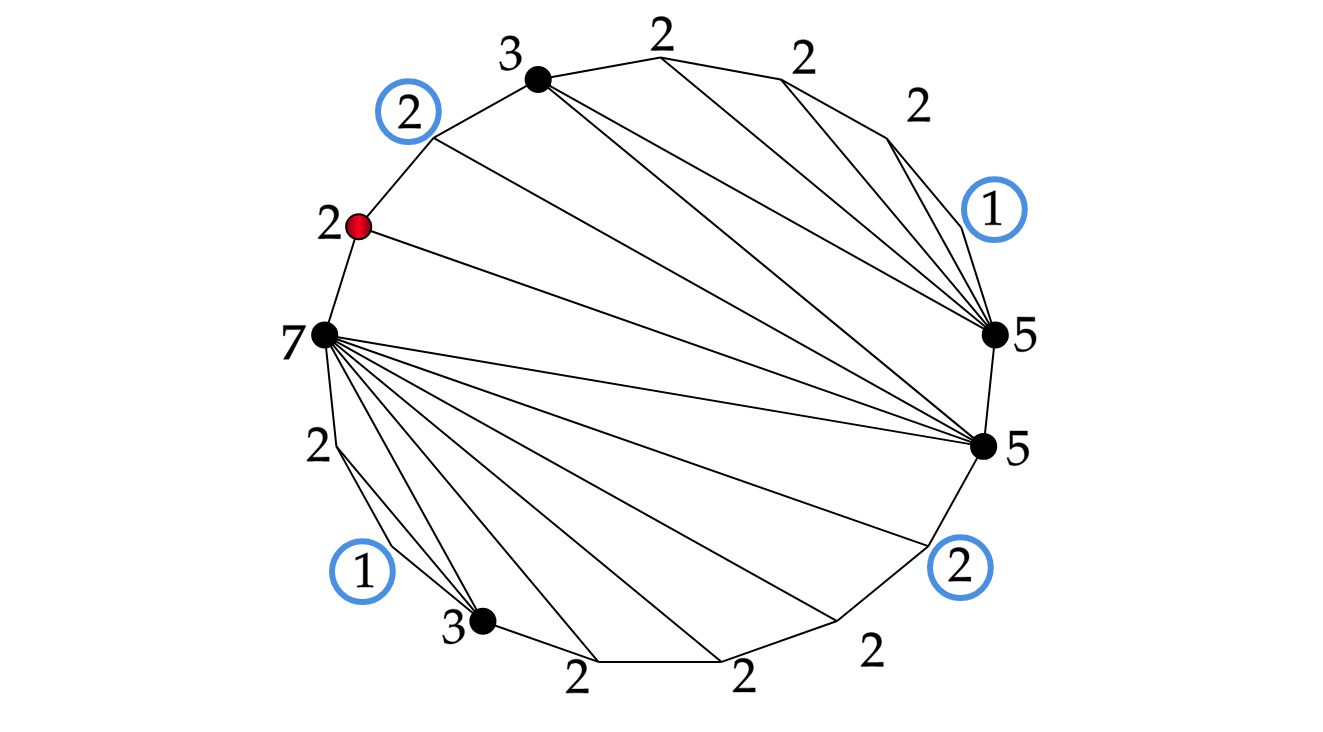}
        \end{tabular}
        \caption{Change of frame for triangulations in Figure \ref{final example}.}
        \label{change of frame, final example}
    \end{figure}
    The frames in Figure \ref{change of frame, final example} determine the extended zero chains
    $$0=[7,2,2,3,2,2,2,1,5,5,2,2,2,2,3,1\mid 2],$$
    $$0=[7,2,1,3,2,2,2,2,5,5,1,2,2,2,3,2\mid 2].$$
    These are WW-decompositions of the H-J continued fraction
    $$\frac{m_1}{m_1-a_1}=\frac{40223}{6425}=[7,2,2,3,2,2,2,2,5,5,2,2,2,2,3,2].$$
    It defines the wormhole singularity $\frac{1}{40223}(1,33798)$. Its minimal resolution is given by the H-J continued fraction
    $$\frac{40223}{33798}=[2,2,2,2,2,5,7,2,2,3,2,2,7,3].$$
\end{example}

\end{document}